\documentclass[11pt,oneside,english,reqno]{amsart}
\usepackage{amstext}
\usepackage{amsthm}
\usepackage{mathptmx}
\usepackage{helvet}
\usepackage{courier}
\usepackage{newtxmath}
\usepackage[T1]{fontenc}
\usepackage[utf8]{inputenc}
\usepackage[letterpaper]{geometry}
\usepackage{babel}
\usepackage{verbatim}
\usepackage{refstyle}
\usepackage{mathtools}
\usepackage[unicode=true,pdfusetitle,
 bookmarks=true,bookmarksnumbered=false,bookmarksopen=false,
 breaklinks=true,pdfborder={0 0 0},pdfborderstyle={},backref=false,colorlinks=false]
 {hyperref}

\makeatletter

\AtBeginDocument{\providecommand\eqref[1]{\ref{eq:#1}}}
\AtBeginDocument{\providecommand\thmref[1]{\ref{thm:#1}}}
\AtBeginDocument{\providecommand\secref[1]{\ref{sec:#1}}}
\AtBeginDocument{\providecommand\propref[1]{\ref{prop:#1}}}
\AtBeginDocument{\providecommand\lemref[1]{\ref{lem:#1}}}
\AtBeginDocument{\providecommand\subsecref[1]{\ref{subsec:#1}}}
\RS@ifundefined{subsecref}
  {\newref{subsec}{name = \RSsectxt}}
  {}
\RS@ifundefined{thmref}
  {\def\RSthmtxt{theorem~}\newref{thm}{name = \RSthmtxt}}
  {}
\RS@ifundefined{lemref}
  {\def\RSlemtxt{lemma~}\newref{lem}{name = \RSlemtxt}}
  {}

\numberwithin{equation}{section}
\numberwithin{figure}{section}
\numberwithin{table}{section}
\theoremstyle{plain}
\newtheorem{thm}{\protect\theoremname}[section]
\theoremstyle{plain}
\newtheorem{lem}[thm]{\protect\lemmaname}
\theoremstyle{plain}
\newtheorem{prop}[thm]{\protect\propositionname}

\usepackage{refstyle}
\newref{rem}{name=Remark~,names=Remarks~}
\newref{lem}{name=Lemma~,names=Lemmas~}
\newref{def}{name=Definition~,names=Definitions~}
\newref{thm}{name=Theorem~,names=Theorems~}
\newref{prop}{name=Proposition~,names=Propositions~}
\newref{cond}{name=Condition~,names=Conditions~}
\newref{cor}{name=Corollary~,names=Corollaries~}
\newref{sec}{name=Section~,names=Sections~}
\newref{fig}{name=Figure~,names=Figures~}
\newref{subsec}{name=Subsection~,names=Subsections~}
\newref{subsubsec}{name=Subsubsection~,names=Subsubsections~}
\newref{eq}{name=,refcmd=(\ref{#1})}
\hypersetup{final}
\makeatletter
\def\@maketitle{%
  \normalfont\normalsize
  \@adminfootnotes
  \@mkboth{\@nx\shortauthors}{\@nx\shorttitle}%
  \global\topskip0\p@\relax %
  \@settitle
  \ifx\@empty\authors \else \@setauthors \fi
  \ifx\@empty\@dedicatory
  \else
    \baselineskip18\p@
    \vtop{\centering{\footnotesize\itshape\@dedicatory\@@par}%
      \global\dimen@i\prevdepth}\prevdepth\dimen@i
  \fi
  \@setabstract
  \normalsize
    \dimen@18\p@ \advance\dimen@-\baselineskip
    \vskip\dimen@\relax
}
\def\@setaddresses{\par
  \nobreak \begingroup
\footnotesize
  \def\author##1{\nobreak\addvspace\bigskipamount}%
  \def\\{\unskip, \ignorespaces}%
  \interlinepenalty\@M
  \def\address##1##2{\begingroup
    \par\addvspace\bigskipamount\noindent
    \@ifnotempty{##1}{(\ignorespaces##1\unskip) }%
    {\scshape\ignorespaces##2}\par\endgroup}%
  \def\curraddr##1##2{\begingroup
    \@ifnotempty{##2}{\nobreak\curraddrname
      \@ifnotempty{##1}{, \ignorespaces##1\unskip}\/:\space
      ##2\par}\endgroup}%
  \def\email##1##2{\begingroup
    \@ifnotempty{##2}{\nobreak\noindent\emailaddrname
      \@ifnotempty{##1}{, \ignorespaces##1\unskip}\/:\space
      \ttfamily##2\par}\endgroup}%
  \def\urladdr##1##2{\begingroup
    \def~{\char`\~}%
    \@ifnotempty{##2}{\nobreak\urladdrname
      \@ifnotempty{##1}{, \ignorespaces##1\unskip}\/:\space
      \ttfamily##2\par}\endgroup}%
  \addresses
  \endgroup
}
\makeatother

\makeatother

\usepackage[style=trad-plain,doi=false,url=false,isbn=false]{biblatex}
\providecommand{\lemmaname}{Lemma}
\providecommand{\propositionname}{Proposition}
\providecommand{\theoremname}{Theorem}

\begin{document}
\global\long\def\Law{\operatorname{Law}}%

\global\long\def\supp{\operatorname{supp}}%

\global\long\def\dif{\mathrm{d}}%

\global\long\def\e{\mathrm{e}}%

\global\long\def\p{\mathrm{p}}%

\global\long\def\q{q}%

\global\long\def\Dif{\mathrm{D}}%

\global\long\def\eps{\varepsilon}%

\global\long\def\Cov{\operatorname{Cov}}%

\global\long\def\Var{\operatorname{Var}}%

\global\long\def\sgn{\operatorname{sgn}}%

\global\long\def\Uniform{\operatorname{Uniform}}%

\global\long\def\Lip{\operatorname{Lip}}%

\global\long\def\argmin{\operatorname*{argmin}}%

\global\long\def\argmax{\operatorname*{argmax}}%

\global\long\def\osc{\operatorname*{osc}}%
\DeclareRobustCommand\lbind{L}
\title{Existence of stationary stochastic Burgers evolutions on \texorpdfstring{$\mathbf{R}^2$}{R\texttwosuperior}
and \texorpdfstring{$\mathbf{R}^3$}{R\textthreesuperior}}
\author{Alexander Dunlap}
\date{October 27, 2019}
\thanks{The author was partially supported by the NSF Graduate Research Fellowship
Program under grant DGE-1147470.}
\address{Department of Mathematics, Stanford University, 450 Jane Stanford
Way, Building 380, Stanford, CA 94305 USA}
\email{{\rmfamily\itshape \href{mailto:ajdunl2@stanford.edu}{\nolinkurl{ajdunl2@stanford.edu}}}}
\begin{abstract}
We prove that the stochastic Burgers equation on $\mathbf{R}^{d}$,
$d<4$, forced by gradient noise that is white in time and smooth
in space, admits spacetime-stationary solutions. These solutions are
thus the gradients of solutions to the KPZ equation on $\mathbf{R}^{d}$
with stationary gradients. The proof works by proving tightness of
the time-averaged laws of the solutions in an appropriate weighted
space.
\end{abstract}

\maketitle

\section{Introduction}

Consider the stochastic Burgers equation
\begin{equation}
\dif u=\tfrac{1}{2}[\Delta u-\nabla(|u|^{2})]\dif t+\dif(\nabla V)\label{eq:Burgers}
\end{equation}
on $\mathbf{R}\times\mathbf{R}^{d}$, where $V$ is a Wiener process
in time with smooth values in space and $\dif$ is the Itô time differential
(so $\dif(\nabla V)$ is white in time and smooth and of gradient
type in space). To be precise, let $\dif W$ be a space-time white
noise on\textbf{ $\mathbf{R}\times\mathbf{R}^{d}$} and let $V=\rho*W$,
where $\rho\in\mathcal{C}^{\infty}(\mathbf{R}^{d})\cap L^{2}(\mathbf{R}^{d})$
is a spatial mollifier and $*$ denotes convolution in space. We will
always work with strong solutions to \eqref{Burgers}, as defined
in \cite{DPZ14}. Our goal is to prove the following.
\begin{thm}
\label{thm:maintheorem}If $d<4$, then there exist statistically
spacetime-stationary strong solutions to \eqref{Burgers}.
\end{thm}

Equation \eqref{Burgers} is related via the \emph{Cole--Hopf transform}\cite{Hop50,Col51}
to the KPZ equation\cite{KPZ86} and multiplicative stochastic heat
equation (SHE). Expicitly, if $\phi$ solves the SHE
\[
\dif\phi=\tfrac{1}{2}\Delta\phi\dif t-\phi\dif V,
\]
then $h=-\log\phi$ solves the KPZ equation
\[
\dif h=\tfrac{1}{2}[\Delta h-|\nabla h|^{2}]\dif t+\dif V,
\]
and $u=\nabla h$ solves \eqref{Burgers}. It is easy to see by a
variance computation, using the Feynman--Kac formula, that when $d=2$,
or when $d\ge3$ and $\|\rho\|_{L^{2}(\mathbf{R})}$ is sufficiently
large, the SHE started at $\phi(t,\cdot)\equiv1$ has diverging pointwise
statistics at $t\to\infty$. This implies the same for KPZ solutions
started at zero. Thus it is expected that these equations, unlike
\eqref{Burgers}, do \emph{not} admit space-time stationary solutions
on the whole space. So informally, \thmref{maintheorem} indicates
that the obstruction to stationary KPZ solutions is the growth of
fluctuations of the zero-frequency mode, which is destroyed by the
gradient. The present work thus extends the theory of long-time KPZ
solution behavior. Long-time/large-space scaling limits for the multidimensional
SHE and KPZ equations, especially phrased in terms of solving the
(renormalized) equations with driving noise $V$ that is white in
space as well as time, have been of substantial recent interest. See
\cite{BC98,CSZ17,CD18,CSZ18,Gu18,GQT19} for some results in $d=2$
and \cite{DS80,TZ98,MU18,GRZ18,DGRZ18,DGRZ18b} for some results in
$d\ge3$.

Problems similar to ours have been considered in one spatial dimension
in \cite{BCK14,Bak16,BL19,DGR19}. We obtain the existence results
of these papers in our two- and three-dimensional settings. The functional-analytic
framework that we use is the same as that of \cite{DGR19}. (We completely
avoid the consideration of Lagrangian minimizers and directed polymers
used in \cite{BCK14,Bak16,BL19} and most other earlier works in this
area.) To show the existence of stationary solutions, we prove that
the time-averaged laws of the solutions are tight in a space in which
the equation is well-posed. We use a simple but fundamental idea of
\cite{DGR19} to bound the pointwise variance of the solution at a
random time using the Cole--Hopf transform and Jensen's inequality.
(See \secref{L2bound} below.) However, the tools used in \cite{DGR19}
to bound the derivatives of the solutions (similar to those used in
the periodic setting in \cite{Bor13}) are not available in dimension
$d\ge2$. Moreover, the growth at infinity that is allowed by a variance
bound in $d\ge2$ is too much for the equation \eqref{Burgers} to
be well-posed. These issues pose serious challenges for proving compactness.

Stationary solutions for \eqref{Burgers} are known to exist in $d\ge3$
when $V$ is multiplied by a sufficiently small constant---the so-called
\emph{weak noise} regime, which does not exist in $d=1,2$. This can
be seen as a consequence of the fact that the stationary solutions
exist for the stochastic heat equation in this regime \cite{DS80,TZ98,MSZ16,DGRZ18},
and was studied directly in \cite{Kif97}. When the noise is stronger,
stationary solutions are thought to exist for the Burgers equation
but not for the SHE. This paper thus extends the existence of stationary
solutions for Burgers in $d=3$ to the strong-noise setting.

In another direction, the present work can be seen as an extension
of some of the results of \cite{IK03,GIKP05,Bor16} about stationary
solutions for the multidimensional stochastic Burgers equation on
a periodic domain. The problem of tightness on the whole space is
very different from on a periodic domain, since on the whole space
Poincaré-type inequalities are not available to control the solution
in terms of its derivatives, and controlling the growth of the solutions
at infinity becomes crucial. However, a key part of our argument is
the Kruzhkov maximum principle\cite{Kru64}, which was introduced
in the context of the periodic stochastic Burgers equation in \cite{Bor16}
to give bounds on the slopes of the rarefactions of the solutions.
Mathematically, this means bounds on $(\partial_{i}u_{i})^{+}$ for
$i=1,\ldots,d$.

The literature on stationary solutions for the stochastic Burgers
equation on a one-dimensional compact domain is by now quite extensive.
We mention for example the seminal paper \cite{EKMS00}. Also relevant
are the more recent results \cite{Bor13,Bor18}, which use PDE-style
techniques closer to those in the present paper.

We do not address questions of uniqueness of the stationary measures,
nor do we prove anything about convergence to them. In \cite{Bor13}
on the one-dimensional torus, and \cite{DGR19} on the line, this
was studied using the $L^{1}$ contraction properties of the one-dimensional
Burgers equation. The Burgers equation in $d\ge2$ does not have an
$L^{1}$ contraction property. On the multidimensional torus, \cite{Bor16}
replaced this property with an $L^{\infty}$ contraction property
for the mean-zero integral of $u$, which solves the KPZ equation
\eqref{KPZ} up to a zero-frequency term. When solutions are unbounded
as they are when the equation is considered on the whole space, it
is not clear if or how this argument can be adapted. Using an approach
based on directed polymers, \cite{BL19} proved convergence to the
stationary solution and a one-force-one-solution principle for the
stochastic Burgers equation on the real line. However, the proof relied
on the ordering of the real line and on a relaxation property for
the Burgers equation coming from the fact that the forcing was discrete-time
``kicks.''

\subsection*{Strategy of the proof and outline of the paper}

The technical heart of our work is the use of a Kruzhkov maximum principle
in a nonperiodic setting. We want to show that, as in \cite{Bor16},
$(\partial_{i}u_{i})^{+}$ can be controlled by the solution to a
Riccati-type equation that brings any initial data down to size of
order $1$ in time of order $1$. Physically, this corresponds to
the fact that rarefaction waves are destroyed by the nonlinearity
in Burgers evolution, in contrast to shocks, which are actually encouraged
by the nonlinearity and are only prevented by the viscosity term.
The situation is immediately much more complicated than that considered
in \cite{Bor16} because we do not expect global-in-space bounds,
and so we must work with the equation for the weighted solution. We
do this \secref{maxprinc}, obtaining the derivative bound \eqref{dtzstarreadyforlemmas},
below.

Since the equation for the maximum of $(\partial_{i}u_{i})^{+}$ also
involves $u$ itself, it is necessary to control $u$ in terms of
$(\partial_{i}u_{i})^{+}$ and \emph{a priori} controlled quantities
to close the argument. Unlike in the mean-$0$ periodic case, a function
on the whole space cannot be controlled by its derivatives. The only
available \emph{a priori} controlled quantity is the pointwise $L^{2}$
bound on $u$ discussed above and proved in \secref{L2bound}. Thus
we need to pass from $L^{2}$ control in space to $\mathcal{C}^{0}$
control in space (which is necessary to close the maximum principle
argument for $(\partial_{i}u_{i})^{+}$), using only a bound on the
positive part of the diagonal terms of the Jacobian of $u$. Our tool
for this is \propref{deterministic-bound}, which is stated for general
differentiable functions $\mathbf{R}^{d}\to\mathbf{R}^{d}$ of gradient
type.

The other key challenge to overcome is that for the Burgers equation
to be well-posed, the initial conditions must grow sublinearly at
infinity. Otherwise, mass from infinity could reach the origin in
finite time. The proof that the Burgers equation is well-posed under
this assumption is the same as that of \cite[Theorem 2.3]{DGR19}
in one spatial dimension. We summarize it in \secref{wellposedness}
below. A variance bound on a space-stationary function $g:\mathbf{Z}^{d}\to\mathbf{R}$
restricts the growth of $g$ to be less than order $|x|^{d/2+\delta}$
for any $\delta>0$. This is insufficient when $d\ge2$. Thus, we
need to ``stretch out'' the mesh on which we sample the values of
$u(t,\cdot)$, by using the bound on $(\partial_{i}u_{i})^{+}$ to
interpolate. (See the proof of \propref{deterministic-bound}.) This
makes $d<4$ tractable.

In \secref{tightness}, we put the bounds together to establish tightness
of the time-averaged laws of the Burgers solutions. Our bounds are
in terms of solutions to Riccati-type equations. Since our \emph{a
priori} control of the solutions established in \secref{L2bound}
is only $L^{2}$ in time, we need to use the well-posedness of the
equation to upgrade this to $\mathcal{C}^{0}$ control of the solutions
on very short time scales. It turns out, however, that these time
scales are long enough to find a time at which a spatial $L^{2}$
norm is not too large, using the $L^{2}$ in time control. Then we
can iterate the argument. This is done in the proofs of \lemref{cutitininhalf}
and \propref{bringitdown}. Once the time-averaged laws of the solutions
are controlled uniformly in time, in a space in which the equation
\eqref{Burgers} is well-posed, showing the existence of stationary
solutions and proving \thmref{maintheorem} is a routine application
of parabolic regularity and a Krylov--Bogoliubov argument.

\subsection*{Notation}

We will often write $x^{+}=\max\{x,0\}$. If $x\in\mathbf{R}^{d}$,
we define $|x|$ to be the Euclidean norm of $x$. For a metric space
$\mathcal{X}$ and a normed space $\mathcal{Y}$, we let $\mathcal{C}(\mathcal{X};\mathcal{Y})$
be the space of \emph{bounded} $\mathcal{Y}$-valued continuous functions
on $\mathcal{X}$, equipped with the supremum norm. We define more
function spaces in \secref{wellposedness} below.%

\subsection*{Acknowledgments}

We thank Cole Graham and Lenya Ryzhik for interesting conversations.
We also thank Felix Otto for pointing out the link between the arguments
of \secref{deterministictheory} and the theory of semiconvexity.

\section{Well-posedness\label{sec:wellposedness}}

In order to discuss the well-posedness theory for the equation \eqref{Burgers},
we first need to introduce weighted function spaces. By a weight we
simply mean a positive function $w:\mathbf{R}^{d}\to\mathbf{R}$.
We will use weights of the form
\begin{align}
\p_{\ell,K}(x) & =(\langle x\rangle+K)^{\ell}, & \p_{\ell} & =\p_{\ell,1}, & \xi(x) & =(\log(\langle x\rangle+1))^{3/4}, & \langle x\rangle & =\sqrt{1+|x|^{2}}.\label{eq:weights}
\end{align}
If $w$ is a weight, $k\in\mathbf{Z}_{\ge0}$, and $\alpha\in(0,1)$,
define the norms
\begin{align*}
\|f\|_{\mathcal{C}_{w}^{k}} & =\adjustlimits\sup_{x\in\mathbf{R}^{d}}\max_{j\in\{0,\ldots,k\}}\frac{|f^{(j)}(x)|}{w(x)}, & \|f\|_{\mathcal{C}_{w}^{\alpha}} & =\max\left\{ \|f\|_{\mathcal{C}_{w}},\sup_{|x-y|\le1}\frac{|f(x)-f(y)|}{w(x)|x-y|^{\alpha}}\right\} ,
\end{align*}
and define $\mathcal{C}_{w}^{\alpha}$ to be the space of functions
for which the $\mathcal{C}_{w}^{\alpha}$ norm is finite. Define $\mathcal{A}_{w}^{\alpha}$
to be the subspace of $\mathbf{R}^{d}$-valued elements of $\mathcal{C}_{w}^{\alpha}$
consisting of functions which are gradients, equipped with the $\mathcal{C}_{w}^{\alpha}$
norm. We will abbreviate $\mathcal{C}_{w}=\mathcal{C}_{w}^{0}$ and
$\mathcal{A}_{w}=\mathcal{A}_{w}^{0}$. We will prove tightness results
in the space $\mathcal{A}_{w}$, which introduces a minor technicality
in that $\mathcal{A}_{w}$ is not a Polish space, so Prokhorov's theorem
does not apply. Thus we also consider the spaces
\begin{equation}
\tilde{\mathcal{A}}_{w}=\left\{ f\in\mathcal{A}_{w}\ :\ \lim_{|x|\to\infty}f(x)=0\right\} ,\label{eq:tildeAw}
\end{equation}
which are separable and thus Polish spaces. The following compactness
lemma is simple and was also used in \cite{DGR19}.
\begin{lem}
\label{lem:weightedcompactness}Let $\alpha>0$ and $\ell,\ell_{1},\ell_{2}$
be such that $\ell_{1}<\ell$. Then the embedding $\mathcal{C}_{\p_{\ell_{1}}}\cap\mathcal{C}_{\p_{\ell_{2}}}^{\alpha}\hookrightarrow\mathcal{C}_{\p_{\ell}}$
is compact.
\end{lem}

We now establish that the equation \eqref{Burgers} is well-posed
in the space $\mathcal{C}_{\p_{\ell}}$ as long as $\ell<1$. As discussed
in the introduction, the restriction to $\ell<1$ is essential. The
proof given for the one-dimensional case in \cite{DGR19} goes through
for general $d$ without essential modification, but we will need
some lemmas from the construction in this section, so we sketch the
proof here.

We will define solutions to periodic problems and then obtain whole-space
solutions as almost-sure limits. Thus we need a sequence of periodic
driving noises that converge to the full-space noise almost surely
in a weighted space.
\begin{lem}
\label{lem:VLdef}We have a sequence of space-stationary Wiener processes
$V^{[L]}$, $L\in[1,\infty)$, so that $V^{[L]}$ agrees in law with
$\rho*W^{[L]}$, where $W^{[L]}$ is an $L\mathbf{Z}^{d}$-periodic
space-time white noise, and for any $T,\ell>0$ we almost surely have
\begin{equation}
\lim_{L\to\infty}\|V^{[L]}-V\|_{\mathcal{C}([0,T];\mathcal{C}_{\p_{\ell}})}=0.\label{eq:VLconvtoV}
\end{equation}
\end{lem}

\begin{proof}
For each $L\in[1,\infty)$, let $\chi^{[L]}\in\mathcal{C}_{\mathrm{c}}^{\infty}$
be a function so that $\supp\chi^{[L]}\subset[-L,L]^{d}$ and $\sum_{\mathbf{j}\in\mathbf{Z}^{d}}\chi^{[L]}(x+L\mathbf{j})^{2}=1$
for all $x\in\mathbf{R}^{d}$. Define $W^{[L]}=\sum_{\mathbf{j}\in\mathbf{Z}^{d}}\tau_{L\mathbf{j}}(\chi_{L}W)$,
where $\tau_{L\mathbf{j}}$ denotes spatial translation by $L\mathbf{j}$,
and define $V^{[L]}=\rho*W^{[L]}$. A simple covariance computation
shows that $V^{[L]}$ has the desired law, and the convergence \eqref{VLconvtoV}
is also clear from the smoothness and Gaussianity of the fields.
\end{proof}
For notational convenience, put $V^{[\infty]}=V$. We note that if
$\psi$ solves the linear problem
\begin{equation}
\dif\psi=\tfrac{1}{2}\Delta\psi\,\dif t+\dif(\nabla V^{[L]})\addtocounter{equation}{1}\tag{\theequation\ensuremath{_{\lbind}}}\label{eq:ASHEL}
\end{equation}
and $v$ (classically) solves the problem
\begin{equation}
\partial_{t}v=\tfrac{1}{2}\Delta v-\tfrac{1}{2}\nabla(|v+\psi|^{2}),\addtocounter{equation}{1}\tag{\theequation\ensuremath{_{\psi}}}\label{eq:vPDE}
\end{equation}
then $u=v+\psi$ solves the problem
\begin{equation}
\dif u=\tfrac{1}{2}[\Delta u-\nabla(|u|^{2})]\dif t+\dif(\nabla V^{[L]}),\addtocounter{equation}{1}\tag{\theequation\ensuremath{_{\lbind}}}\label{eq:uLPDE}
\end{equation}
which is \eqref{Burgers} with the periodized noise $V^{[L]}$ in
place of the whole-space noise $V$. The stochastic PDE \eqref{ASHEL}
has Gaussian solutions, while \eqref{vPDE} is a PDE with classical
solutions. Thus we will generally work with \eqref{ASHEL} and \eqref{vPDE},
and add the solutions together to obtain solutions to \eqref{uLPDE}.

The fact that, for any $L\mathbf{Z}^{d}$-periodic $\underline{u}\in\mathcal{C}$,
there exist global-in-time solutions to \eqref{uLPDE} with initial
condition $u(0,\cdot)=\underline{u}$ can be proved by a fixed-point
argument, as was done in \cite{Bor16}. (Similar arguments were carried
out in \cite{DPDT94,DPG94,Bor13,DGR19}.) The global-in-time well-posedness
theory on the whole space is somewhat more involved, since appropriate
weights must be used: certainly if the initial conditions are to live
in $\mathcal{C}_{\p_{\ell}}$, then $\ell$ must be required to be
strictly less than $1$, or else mass from infinity could accumulate
at the origin in finite time.
\begin{thm}
\label{thm:existenceuniqueness}Fix $0<m<m'<1$ and $M,R,T<\infty$.
There is a random constant $X=X(m,m',M,R,T)$, finite almost surely,
and an $\alpha>0$ so that the following holds. If $L\in(0,\infty]$
and $\underline{u}\in\mathcal{A}_{\p_{m}}$ is $L\mathbf{Z}^{d}$-periodic,
and $\|\underline{u}\|_{\mathcal{C}_{\p_{m}}}\le M$, then there is
a unique $u\in\mathcal{C}([0,T];\mathcal{A}_{\p_{m'}})$ with 
\begin{equation}
\|u\|_{\mathcal{C}([0,T];\mathcal{C}_{\p_{m'}})},\|u\|_{\mathcal{C}([1,T];\mathcal{C}_{\p_{2m'}}^{\alpha})}\le X\label{eq:ubounds}
\end{equation}
and so that $u(0,\cdot)=\underline{u}$ and $u$ is a strong solution
to \eqref{uLPDE}. Moreover, if $L'\in(0,\infty]$, $\tilde{\underline{u}}\in\mathcal{A}_{\p_{m}}$
is $L'\mathbf{Z}^{d}$-periodic and such that $\|\underline{\tilde{u}}\|_{\mathcal{C}_{\p_{\ell}}}\le M$,
and $\tilde{u}$ is the strong solution to {\renewcommand\lbind{L'}\eqref{uLPDE}}
with initial condition $\tilde{u}(0,\cdot)=\underline{\tilde{u}}$,
then
\begin{equation}
\|u-\tilde{u}\|_{\mathcal{C}([0,T];\mathcal{C}([-R,R]^{d}))}\le X[\|\underline{u}-\underline{\tilde{u}}\|_{\mathcal{C}_{\p_{\ell}}}+\|\psi^{[L]}-\psi^{[L']}\|_{\mathcal{C}([0,T];\mathcal{C}_{\p_{m}})}].\label{eq:continuity}
\end{equation}
\end{thm}

Since the proof of \thmref{existenceuniqueness} is a straightforward
generalization of the proofs given in \cite[Section 2]{DGR19}, we
only sketch it. The existence theory for the periodic problem is a
standard fixed-point argument along the lines of those discussed in
\cite{DPDT94,DPZ14,Bor16}. The fact that $\|u\|_{\mathcal{C}([0,T];\mathcal{C}_{\p_{m'}})}$
is bounded uniformly in $L$ (the first part of \eqref{ubounds})
follows from a maximum principle argument along the lines of \cite[Proposition 2.10]{DGR19},
which we establish below as \propref{ugrowth}, both because we will
use it in the sequel and because its proof statement requires attention
to the fact that the solution $u$ is assumed to be of gradient type,
which is a vacuous assumption in $d=1$. Finally, the existence theory
for the nonperiodic problem, the bound on $\|u\|_{\mathcal{C}([1,T];\mathcal{C}_{\p_{2m'}}^{\alpha})}$
in \eqref{ubounds}, and \eqref{continuity} follow from parabolic
regularity estimates essentially identical to those given in \cite{DGR19}.

We now establish uniform-in-$L$ bounds on $\|u\|_{\mathcal{C}([0,T];\mathcal{C}_{\p_{m'}})}$.
Our bound in the following proposition, which was similarly proved
in the one-dimensional case in \cite{DGR19}, shows that $\|u(t,\cdot)\|_{\mathcal{C}_{\p_{m',K}}}$
grows in $t$ at most like the solution to a Riccati equation of the
form $f'=cf^{2}$, where the constant $c$ can be made small by choosing
$K$ large relative to the size of the noise. Increasing $K$ reduces
the maximum gradient of $\p_{m',K}$, attenuating the effect on $\|u(t,\cdot)\|_{\mathcal{C}_{\p_{m',K}}}$
of movement of mass in the solution $u(t,\cdot)$. This movement of
mass is all that we have to be concerned about, as the addition of
mass coming from the noise is finite in the weighted norm, and the
nonlinearity in \eqref{Burgers} is a gradient so does not add mass.
\begin{prop}
\label{prop:ugrowth}For any $m\in(0,1)$, $T\in(0,\infty)$, and
$\eps\in(0,(1-m)/T)$, there is a constant $C=C(m,T,\eps)<\infty$
so that the following holds. Let $L\in[1,\infty)$, $\psi\in\mathcal{C}([0,1];\mathcal{A}_{\xi}^{1})$,
$v\in\mathcal{C}([0,1];\mathcal{A}_{\p_{m,K}})$ be a solution to
\eqref{vPDE}, and $u=v+\psi$. (Recall the definition \eqref{weights}
of $\xi$.) There is a constant $K_{0}=K_{0}(d,\eps,\|\psi\|_{\mathcal{C}([0,T];\mathcal{C}_{\xi}^{1})})<\infty$
so that if $K\ge K_{0}$, then for all
\[
t\in[0,K^{1-\ell}\min\{\|u(0,\cdot)\|_{\mathcal{C}_{\p_{m,K}}}^{-1},1\}],
\]
we have
\begin{equation}
\|u(t,\cdot)\|_{\mathcal{C}_{\p_{m+\eps t,K}}}\le[\min\{\|u(0,\cdot)\|_{\mathcal{C}_{\p_{m,K}}}^{-1},1\}-K^{-(1-\ell)}t]^{-1}+\|\psi\|_{\mathcal{C}([0,T];\mathcal{C}_{\p_{m,K}})}.\label{eq:ubound}
\end{equation}
\end{prop}

The proof of \propref{ugrowth} will be based on the maximum principle
for the weighted solution $v$, so first we show what happens when
we differentiate a weighted version of $v$.
\begin{lem}
\label{lem:differentiateformaxprinciple}Suppose $\psi:\mathbf{R}\times\mathbf{R}^{d}\to\mathbf{R}^{d}$
is differentiable and a gradient, $v:\mathbf{R}\times\mathbf{R}^{d}\to\mathbf{R}^{d}$
is a gradient and solves \eqref{vPDE} for some $L\in(0,\infty]$,
$a:\mathbf{R}\times\mathbf{R}^{d}\to\mathbf{R}$ is differentiable
in time and twice-differentiable in space, and $\gamma=av$. Then
we have, for each $i\in\{1,\ldots,d\}$, that
\begin{equation}
\begin{aligned}\partial_{t}\gamma_{i} & =\gamma_{i}\partial_{t}(\log a)+\tfrac{1}{2}\Delta\gamma_{i}-\nabla(\log a)\cdot\nabla\gamma_{i}+\tfrac{1}{2}\gamma_{i}[|\nabla(\log a)|^{2}-\Delta(\log a)]\\
 & \qquad-(a^{-1}\gamma+\psi)\cdot(\nabla\gamma_{i}-\gamma_{i}\nabla(\log a))-(\gamma+a\psi)\cdot\partial_{i}\psi.
\end{aligned}
\label{eq:dtyi-final}
\end{equation}
\end{lem}

\begin{proof}
By the product rule we have
\[
\partial_{t}\gamma=v\partial_{t}a+a\partial_{t}v=\gamma\partial_{t}(\log a)+\frac{a}{2}[\Delta v-\nabla(|v+\psi|^{2})].
\]
It is a general fact about functions $f$, $g$, and $h=fg$ that
\begin{align}
f\nabla g & =\nabla h-h\nabla(\log f); & f\Delta g & =\Delta h-2\nabla(\log f)\cdot\nabla h+h[|\nabla(\log f)|^{2}-\Delta(\log f)].\label{eq:calculusidentities}
\end{align}
Therefore, we have
\begin{equation}
\partial_{t}\gamma_{i}=\gamma_{i}\partial_{t}(\log a)+\tfrac{1}{2}\Delta\gamma_{i}-\nabla(\log a)\cdot\nabla\gamma_{i}+\tfrac{1}{2}\gamma_{i}[|\nabla(\log a)|^{2}-\Delta(\log a)]-\tfrac{1}{2}a\partial_{i}(|v+\psi|^{2}).\label{eq:dtyi}
\end{equation}
On the other hand, using the fact that $v+\psi$ is a gradient and
also \eqref{calculusidentities} again, we have
\begin{align*}
\tfrac{1}{2}a\partial_{i}(|v+\psi|^{2})=a(v+\psi)\cdot\partial_{i}(v+\psi)=a(v+\psi)\cdot\nabla(v_{i}+\psi_{i}) & =(a^{-1}\gamma+\psi)\cdot(\nabla\gamma_{i}-\gamma_{i}\nabla(\log a))+(\gamma+a\psi)\cdot\partial_{i}\psi,
\end{align*}
Plugging this into \eqref{dtyi}, we obtain \eqref{dtyi-final}.
\end{proof}
Now we are ready to prove \propref{ugrowth}.
\begin{proof}[Proof of \propref{ugrowth}.]
Let $a(t,x)=1/\p_{m+\eps t,K}(x)$ and put $\gamma=av$. %
Multiplying \eqref{dtyi-final} (with this choice of $a$) by $\gamma_{i}$
and summing over $i$, we obtain
\begin{align*}
\tfrac{1}{2}\partial_{t}|\gamma|^{2} & =|\gamma|^{2}\partial_{t}(\log a)+\tfrac{1}{2}\Delta|\gamma|^{2}-|\nabla\gamma|^{2}-\tfrac{1}{2}\nabla(\log a)\cdot\nabla|\gamma|^{2}+\tfrac{1}{2}|\gamma|^{2}[|\nabla(\log a)|^{2}-\Delta(\log a)]\\
 & \qquad-(a^{-1}\gamma+\psi)\cdot[\tfrac{1}{2}\nabla|\gamma|^{2}-|\gamma|^{2}\nabla(\log a)]-(\gamma+a\psi)\cdot(\gamma\cdot\nabla)\psi.
\end{align*}
At a local maximum of $|\gamma|^{2}(t,\cdot)$, we have $\Delta|\gamma|^{2}\le0$
and $\nabla|\gamma|^{2}=0$, so at such a local maximum, using at
the Cauchy--Schwarz inequality we obtain
\begin{equation}
\begin{aligned}\tfrac{1}{2} & \partial_{t}|\gamma|^{2}\le|\gamma|^{2}\partial_{t}(\log a)+\tfrac{1}{2}|\gamma|^{2}[|\nabla(\log a)|^{2}-\Delta(\log a)]+(a^{-1}|\gamma|+|\psi|)|\gamma|^{2}|\nabla(\log a)|+(|\gamma|+a|\psi|)|\gamma||\nabla\psi|\\
 & =a^{-1}|\nabla(\log a)||\gamma|^{3}+|\gamma|^{2}[\tfrac{1}{2}\partial_{t}(\log a)+\tfrac{1}{2}|\nabla(\log a)|^{2}-\tfrac{1}{2}\Delta(\log a)+|\psi||\nabla(\log a)|+|\nabla\psi|+a|\gamma|^{-1}|\psi||\nabla\psi|].
\end{aligned}
\label{eq:dtgamma2}
\end{equation}
Using the simple bounds
\begin{align}
|\nabla(\log a)(t,x)| & \le1, & |\Delta(\log a)| & \le d+2, & |a^{-1}\nabla(\log a)| & \le K^{-(1-\ell)}, & \partial_{t}(\log a)(t,x) & =-\eps\log(\langle x\rangle+K),\label{eq:abounds}
\end{align}
we can derive from \eqref{dtgamma2} that%
, at a local maximum $x$ of $|\gamma|^{2}(t,\cdot)$ such that $|\gamma(t,x)|^{2}\ge1$,
\[
\tfrac{1}{2}\partial_{t}|\gamma|^{2}\le K^{-(1-\ell)}|\gamma|^{3}+|\gamma|^{2}[-\eps\log(\langle x\rangle+K)+\tfrac{1}{2}(d+3)+2\xi\|\psi\|_{\mathcal{C}([0,T];\mathcal{C}_{\xi}^{1})}+a\xi^{2}\|\psi\|_{\mathcal{C}([0,T];\mathcal{C}_{\xi}^{1})}^{2}].
\]
Now we assume $K$ is so large (depending on $d$, $\eps$ and $\|\psi\|_{\mathcal{C}([0,T];\mathcal{C}_{\xi}^{1})}$)
that the term in brackets is guaranteed to be negative regardless
of $t$ and $x$. Then, at a local maximum $x$ of $|\gamma|^{2}(t,\cdot)$
such that $|\gamma(t,x)|^{2}\ge1$, we have $\partial_{t}|\gamma|^{2}\le2K^{-(1-\ell)}|\gamma|^{3}$.
By \cite[Lemma 3.5]{Ham86}, this implies that if we define 
\begin{equation}
\Gamma(t)=\max_{x\in\mathbf{R}^{d}}|\gamma(t,x)|^{2}=\|v(t,\cdot)\|_{\mathcal{C}_{\p_{m+\eps t,K}}}^{2},\label{eq:Gammadef}
\end{equation}
then $\Gamma:[0,T]\to\mathbf{R}$ is Lipschitz and we have $\Gamma'(t)\le2K^{-(1-\ell)}\Gamma(t)^{3/2}$
whenever $\Gamma(t)\ge1$, which means that
\[
\Gamma(t)^{1/2}\le[\min\{\Gamma(0)^{-1/2},1\}-K^{-(1-\ell)}t]^{-1}
\]
whenever $t<K^{1-\ell}\min\{\Gamma(0)^{-1/2},1\}$. Recalling \eqref{Gammadef},
this implies \eqref{ubound}.
\end{proof}

\section{Pointwise variance bound\label{sec:L2bound}}

In the one-dimensional setting, sufficient compactness of solutions
to \eqref{uLPDE} to obtain existence of stationary solutions was
obtained in \cite{DGR19} from $L^{2}$ bounds on the solution and
its derivative at a point. In dimension $d>1$, such an $L^{2}$ bound
on the derivative is not available due to contributions from terms
of the form $\partial_{i}u_{j}$ for $i\ne j$. However, we do have
the same $L^{2}$ bound on the solution at a point. The below proposition
and its proof are the same as in the one-dimensional case \cite{DGR19},
although we need to obtain bounds in the $L\mathbf{Z}^{d}$-periodic
setting, uniformly in $L$.
\begin{prop}
\label{prop:L2bound}There is a constant $C=C(\rho)<\infty$, depending
only on $\rho$, so that if $L\in[1,\infty]$ and $u$ solves \eqref{uLPDE}
with initial condition $u(0,\cdot)\equiv0$, then for all $t\ge0$
and $x\in\mathbf{R}^{d}$, we have the bound
\begin{equation}
\frac{1}{T}\int_{0}^{T}\mathbf{E}u(t,x)^{2}\,\dif t\le C.\label{eq:timeavgedL2bound}
\end{equation}
\end{prop}

\begin{proof}
Let $\phi$ solve the multiplicative stochastic heat equation $\dif\phi=\frac{1}{2}\Delta\phi\,\dif t-\phi\,\dif V^{[L]}$
with $\phi(0,\cdot)\equiv1$, so by the Itô formula, $h=-\log\phi$
solves the KPZ equation
\begin{equation}
\dif h=\frac{1}{2}[\Delta h-|\nabla h|^{2}]\,\dif t+\dif V^{[L]}+\sum_{\mathbf{k}\in\mathbf{Z}^{d}}\rho^{*2}(L\mathbf{k})\label{eq:KPZ}
\end{equation}
with $h(0,\cdot)\equiv0$. By Jensen's inequality and the Itô formula,
$\mathbf{E}h(t,x)\ge-\log\mathbf{E}\phi(t,x)=0$ for all $t\ge0$,
$x\in\mathbf{R}$. Therefore,
\[
0\le\mathbf{E}h(T,x)=\int_{0}^{T}\mathbf{E}\left[-|\nabla h(t,x)|^{2}+\sum_{\mathbf{k}\in\mathbf{Z}^{d}}\rho^{*2}(L\mathbf{k})\right]\,\dif t=T\sum_{\mathbf{k}\in\mathbf{Z}^{d}}\rho^{*2}(L\mathbf{k})-\int_{0}^{T}\mathbf{E}|u(t,x)|^{2}\,\dif t
\]
by \eqref{KPZ} and the space-stationarity of $h$, %
and $\sum_{\mathbf{k}\in\mathbf{Z}^{d}}\rho^{*2}(L\mathbf{k})$ is
bounded above independently of $L$, so we get \eqref{timeavgedL2bound}.
\end{proof}

\section{The Kruzhkov maximum principle\label{sec:maxprinc}}

Now we deploy a version of the Kruzhkov maximum principle \cite{Kru64}.
This type of argument was used for the stochastic Burgers equation
on the $d$-dimensional torus in \cite[Theorem 4.2]{Bor16}. In the
whole-space setting, the need to introduce weights complicates the
situation enormously. In this section, we perform the derivative computations
that will be necessary to set up the argument. The next several sections
will develop the necessary bounds on the terms.
\begin{lem}
\label{lem:dtzbound1}Suppose that $a:\mathbf{R}^{d}\to\mathbf{R}$
is twice-differentiable, $\psi:\mathbf{R}\times\mathbf{R}^{d}\to\mathbf{R}$
is twice-differentiable, and $v$ solves \eqref{vPDE}. Define $u=v+\psi$.
Fix $i\in\{1,\ldots,d\}$ and define $z=a\partial_{i}v_{i}$. Then
\begin{equation}
\partial_{t}z\le\tfrac{1}{2}\Delta z-\nabla(\log a)\cdot\nabla z-\tfrac{1}{2}z[|\nabla(\log a)|^{2}-\Delta(\log a)]-a^{-1}|z+a\partial_{i}\psi_{i}|^{2}-u\cdot(\nabla z-z\nabla(\log a)+a\partial_{ii}\psi).\label{eq:dtzbound1}
\end{equation}
\end{lem}

\begin{proof}
Without loss of generality, we can assume that $i=1$. Define $q=\partial_{1}v_{1}$.
Taking derivatives in \eqref{vPDE}, we have
\begin{align*}
\partial_{t}q=\tfrac{1}{2}\Delta q-\tfrac{1}{2}\partial_{11}(|v+\psi|)^{2} & =\tfrac{1}{2}\Delta q-|\partial_{1}(v+\psi)|^{2}-u\cdot\partial_{11}(v+\psi)\le\tfrac{1}{2}\Delta q-|q+\partial_{1}\psi_{1}|^{2}-u\cdot(\nabla q+\partial_{11}\psi),
\end{align*}
where in the last inequality we used that $u\cdot\partial_{11}v=u\cdot\nabla q$
since $v$ is a gradient. Now $z=aq$, so
\begin{equation}
\partial_{t}z=a\partial_{t}q\le\tfrac{1}{2}a\Delta q-a|q+\partial_{1}\psi_{1}|^{2}-u\cdot(a\nabla q+a\partial_{11}\psi).\label{eq:dtzfirst}
\end{equation}
Applying the two calculus identities \eqref{calculusidentities} to
\eqref{dtzfirst} yields \eqref{dtzbound1} in the case $i=1$:
\[
\partial_{t}z\le\tfrac{1}{2}\Delta z-\nabla(\log a)\cdot\nabla z-\tfrac{1}{2}z[|\nabla(\log a)|^{2}-\Delta(\log a)]-a^{-1}|z+a\partial_{1}\psi_{1}|^{2}-u\cdot(\nabla z-z\nabla(\log a)+a\partial_{11}\psi).\qedhere
\]
\end{proof}
Now we specialize to the case when the partial derivative is evaluated
at a sufficiently large local maximum. The arguments in \secref[s]{deterministictheory}
below will be motivated by the need to understand solutions to the
key differential inequality \eqref{dtzstarreadyforlemmas}, proved
in the following proposition.
\begin{prop}
\label{prop:dtzbound2}Let $\ell,\eps\in(0,1)$ and $i\in\{1,\ldots,d\}$,
and fix $\psi,v,u,z$ as in the statement of \lemref{dtzbound1}.
If $t>0$, $x_{*}$ is a local maximum of $z(t,\cdot)$, and
\begin{equation}
z(t,x_{*})\ge4(d+3)+\|\psi(t,\cdot)\|_{\mathcal{C}_{\p_{\eps}}^{2}},\label{eq:zstarbigenough}
\end{equation}
then 
\begin{equation}
\partial_{t}z(t,x_{*})\le-\frac{1}{8}(\langle x_{*}\rangle+K)^{\ell}z(t,x_{*})^{2}+|u(t,x_{*})|\left(\frac{z(t,x_{*})}{\langle x_{*}\rangle+K}+\frac{\|\psi(t,\cdot)\|_{\mathcal{C}_{\p_{\eps}}^{2}}}{(\langle x_{*}\rangle+K)^{\ell-\eps}}\right).\label{eq:dtzstarreadyforlemmas}
\end{equation}
\end{prop}

\begin{proof}
Define $a(x)=(\langle x\rangle+K)^{-\ell}$. %
{} As in \eqref{abounds}, it is simple to estimate that
\begin{align}
|\nabla(\log a)(x)| & \le(\langle x\rangle+K)^{-1}\le1; & |\Delta(\log a)(x)| & \le d+2.\label{eq:logaderivbounds}
\end{align}
Note also that 
\begin{equation}
a(x_{*})^{\eps/\ell}|\partial_{i}\psi_{i}(t,x_{*})|\le\sup_{x\in\mathbf{R}^{d}}\frac{|\partial_{i}\psi_{i}(t,x)|}{(\langle x\rangle+K)^{\eps}}\le\sup_{x\in\mathbf{R}^{d}}\frac{|\partial_{i}\psi_{i}(t,x)|}{\langle x\rangle^{\eps}}\le\|\psi(t,\cdot)\|_{\mathcal{C}_{\p_{\eps}}^{2}}.\label{eq:axepskkpsi}
\end{equation}
Considering \eqref{dtzbound1}, noting that $\nabla z(t,x_{*})=0$
and $\Delta z(t,x_{*})\le0$ since $x_{*}$ is a local maximum of
$z$, and applying \eqref{logaderivbounds} and \eqref{axepskkpsi},
we obtain
\begin{equation}
\partial_{t}z(t,x_{*})\le\tfrac{1}{2}(d+3)|z(t,x_{*})|-a(x_{*})^{-1}|(z(t,x_{*})-\|\psi\|_{\mathcal{C}_{\p_{\eps}}^{2}})^{+}|^{2}+|u(t,x_{*})|\left(\frac{\ell|z(t,x_{*})|}{\langle x\rangle+K}+a(x_{*})^{1-\eps/\ell}\|\psi(t,\cdot)\|_{\mathcal{C}_{\p_{\eps}}^{2}}\right).\label{eq:dtzxstarreadytosimplify}
\end{equation}
The assumption \eqref{zstarbigenough} yields the inequalities $(z(t,x_{*})-\|\psi\|_{\mathcal{C}_{\p_{\eps}}^{2}})^{+}\ge\tfrac{1}{2}z(t,x_{*})$
and $\tfrac{1}{2}(d+3)|z(t,x_{*})|\le\tfrac{1}{8}z(t,x_{*})^{2}$,
and plugging these and the assumption that $\ell<1$ into \eqref{dtzxstarreadytosimplify}
implies \eqref{dtzstarreadyforlemmas}.
\end{proof}

\section{\label{sec:deterministictheory}The Sobolev-type inequality}

The leading $-z^{2}$ term in \eqref{dtzstarreadyforlemmas}, intuitively,
should bring the positive part of $z$ down to a quantity of order-$1$
in time of order-$1$, independently of the initial conditions. In
order to make this rigorous, of course, we need to control the second
term of \eqref{dtzstarreadyforlemmas}, and in particular control
$u(t,x_{*})$. Besides $z(t,x_{*})$, the only quantities available
to control $u(t,x_{*})$ are $L^{2}$ norms of $u(t,\cdot)$ with
respect to finite measures, whose second moments can be \emph{a priori}
controlled by \propref{L2bound} and Fubini's theorem. In this section,
we prove a bound for general functions $u$ of gradient type which
will be sufficient for our purposes. It is closely related to the
theory of semiconvexity (see e.g. \cite{CS04}) but we give a self-contained
proof. If $f:\mathbf{R}^{d}\to\mathbf{R}$ is twice-differentiable,
define $\Delta^{+}f=\sum_{i=1}^{d}(\partial_{ii}f)^{+}$. The goal
of this section is to prove the following proposition, which we will
apply later on with $\nabla f=u(t,\cdot)$.
\begin{prop}
\label{prop:deterministic-bound}Let $m>0$ and $\alpha>1$ satisfy
$m>d/(2\alpha)$, and put $\ell=m-1+1/\alpha$. There is a constant
$C<\infty$ and a finite measure $\nu$ on $\mathbf{R}^{d}$ so that
if $f:\mathbf{R}^{d}\to\mathbf{R}$ is twice-differentiable and $K\in[0,\infty)$,
then
\begin{equation}
\|\nabla f\|_{\mathcal{C}_{\p_{m,K}}}\le C(\|\Delta^{+}f\|_{\mathcal{C}_{\p_{\ell,K}}}+\|\nabla f\|_{L^{2}(\mathbf{R}^{d};\nu)}),\label{eq:deterministicbound}
\end{equation}
in which we use the notation $\|f\|_{L^{2}(\mathbf{R}^{d};\nu)}=\left(\int_{\mathbf{R}^{d}}|f(x)|^{2}\,\nu(\dif x)\right)^{1/2}.$
\end{prop}

\propref{deterministic-bound} follows from the following local version.
For $x\in\mathbf{R}^{d}$ and $r>0$, define the cube $Q(x,r)=\{y\in\mathbf{R}^{d}\;:\;|x-y|_{\ell^{\infty}}\le r\}$,
where $|x|_{\ell^{\infty}}=\max\limits _{i\in\{1,\ldots d\}}|x_{i}|$,
and the grid $\ddddot{Q}(x,r)=Q(x,r)\cap[x+(\tfrac{1}{3}r+\frac{2}{3}r\mathbf{Z})^{d}]$
of $4^{d}$ points in a cubic lattice in $Q(x,r)$.
\begin{prop}
\label{prop:finallocalest}If $f:\mathbf{R}^{d}\to\mathbf{R}$ is
twice-differentiable, $x\in\mathbf{R}^{d}$, and $r\in(0,\infty)$,
then
\[
\|\nabla f\|_{\mathcal{C}(Q(x,r))}\le84d(6r\|\Delta^{+}f\|_{\mathcal{C}(Q(x,9r))}+\|\nabla f\|_{\mathcal{C}(\ddddot{Q}(x,9r))}).
\]
\end{prop}

Before proving \propref{finallocalest}, we show how it implies \propref{deterministic-bound}.
\begin{proof}[Proof of \propref{deterministic-bound}.]
For $k=0,1,2,\ldots$, let $A_{k}=\{x\in\mathbf{R}^{d}\;:\;|x|\in[k^{\alpha},(k+1)^{\alpha}]\}$,
and let $x_{k,1},\ldots,x_{k,m_{k}}\in\mathbf{R}^{d}$ be chosen so
that the cubes $Q_{k,j}\coloneqq Q(x_{k,j},k^{\alpha-1}),$ $j=1,\ldots,m_{k}$,
cover $A_{k}$ and $m_{k}\le Ck^{d-1}$ for a constant $C$ independent
of $k$. Define $Q'_{k,j}=Q(x_{k,j},9k^{\alpha-1})$, $\ddddot{Q}'_{k,j}=\ddddot{Q}(x_{k,j},9k^{\alpha-1})$,
and 
\[
\nu=\sum_{k=0}^{\infty}\frac{1}{\p_{m}(k^{\alpha})^{2}}\sum_{j=0}^{m_{k}}\sum_{y\in\ddddot{Q}'_{k,j}}\delta_{y},
\]
where $\delta_{y}$ denotes a Dirac $\delta$ measure at $y$. We
note that 
\[
\nu(\mathbf{R}^{d})=\sum_{k=0}^{\infty}\frac{1}{\p_{m}(k^{\alpha})^{2}}\sum_{j=0}^{m_{k}}\sum_{y\in\ddddot{Q}'_{k,j}}1\le\sum_{k=0}^{\infty}\frac{1}{\p_{m}(k^{\alpha})^{2}}\sum_{j=0}^{m_{k}}4^{d}\le C\sum_{k=0}^{\infty}\langle k\rangle^{-2\alpha m+d-1}<\infty,
\]
for some constant $C$ (depending on the dimension $d$), where the
penultimate inequality comes from the bound $m_{k}\le Ck^{d-1}$,
and the fact that the final sum is finite comes from the assumption
$m>d/(2\alpha)$.

By \propref{finallocalest}, we have for each $k,j$ that
\[
\|\nabla f\|_{\mathcal{C}(Q_{k,j})}\le84d\left(6k^{\alpha-1}\|\Delta^{+}f\|_{\mathcal{C}(Q'_{k,j})}+\|\nabla f\|_{\mathcal{C}(\ddddot{Q}'_{k,j})}\right),
\]
so
\begin{align*}
\frac{\|\nabla f\|_{\mathcal{C}(Q(x_{k,j},k^{\alpha-1}))}}{\p_{m,K}(k^{\alpha})} & \le84d\left(6k^{\alpha-1}\frac{\|\Delta^{+}f\|_{\mathcal{C}(Q'_{k,j})}}{\p_{m,K}(k^{\alpha})}+\frac{\|\nabla f\|_{\mathcal{C}(\ddddot{Q}'_{k,j})}}{\p_{m}(k^{\alpha})}\right)\le C\left(\|\Delta^{+}f\|_{\mathcal{C}_{\p_{m-1+1/\alpha,K}}}+\|\nabla f\|_{L^{2}(\mathbf{R}^{d};\nu)}\right)
\end{align*}
for some $C<\infty$ depending only on $d$. Taking a supremum over
all $k,j$, we obtain \eqref{deterministicbound}.
\end{proof}
We will spend the rest of this section proving \propref{finallocalest}.
Define, for $f:\mathbf{R}^{d}\to\mathbf{R}$ and $S\subset\mathbf{R}^{d}$,
the \emph{oscillation} $\osc(f;S)=\sup_{S}f-\inf_{S}f$. Our proof
relies on the observation that if the oscillation of $f$ is controlled
on the set of four points lying on a line, and we have a one-sided
bound on the second derivative of $f$ in the direction of the line,
then we can obtain a bound on the oscillation of $f$ on the segment
connecting the middle two points, and also on the directional derivative
of $f$ in the direction of the line on this segment. We make this
precise in \lemref[s]{onedimlderiv}~and~\ref{lem:onedimlosc} below.
Thus we will seek to control the oscillation of a multidimensional
function $f$ on a cube, which we do by induction on the dimension,
successively extruding a controlled set in each coordinate direction.
First we establish the following elementary lemma about oscillation.
\begin{lem}
\label{lem:combinemanyoscs}If $S_{0}\subset\mathbf{R}^{d}$ and $(S_{i})_{i\in I}$
is a family of subsets of $\mathbf{R}^{d}$, indexed by some index
set $I\not\ni0$, and $S_{0}\cap S_{i}\ne\emptyset$ for each $i\in I$,
and we put $S=S_{0}\cup\bigcup_{i\in I}S_{i}$, then we have
\begin{equation}
\osc(f;S)\le\osc(f;S_{0})+2\sup_{i\in I}\osc(f;S_{i}).\label{eq:combinemanyoscs}
\end{equation}
\end{lem}

\begin{proof}
Let $x_{1},x_{2}\in S$. Choose $i_{1},i_{2}\in\{0\}\cup I$ so that
$x_{i_{j}}\in S_{i_{j}}$, $i=1,2$, and select $y_{j}\in S_{0}\cap S_{i_{j}}$
arbitrarily. Then we have
\[
|f(x_{1})-f(x_{2})|\le|f(x_{1})-f(y_{1})|+|f(y_{1})-f(y_{2})|+|f(y_{2})-f(y_{2})|\le\osc(f;S_{i_{1}})+\osc(f;S_{0})+\osc(f;S_{i_{2}}).
\]
This implies \eqref{combinemanyoscs}.
\end{proof}

\subsection{One-dimensional inequalities}

Before we carry out our induction procedure in \subsecref{dimensioninduction},
we prove the necessary statements for functions of one variable. The
following inequalities are all simple consequences of the fundamental
theorem of calculus.
\begin{lem}
\label{lem:boundwithderivtop}Suppose that $f:[0,1]\to\mathbf{R}$
is differentiable. Then we have
\begin{equation}
\|f\|_{\mathcal{C}([0,1])}\le\max\{|f(0)|,|f(1)|\}+\|(f')^{+}\|_{\mathcal{C}([0,1])}.\label{eq:boundwithderivtop}
\end{equation}
\end{lem}

\begin{proof}
We have, for each $x\in[0,1]$, that
\begin{align*}
f(1)-\int_{x}^{1}(f'(y))^{+}\,\dif y\le f(1)-\int_{x}^{1}f'(y)\,\dif y=f(x) & =f(0)+\int_{0}^{x}f'(y)\,\dif y\le f(0)+\int_{0}^{x}(f'(y))^{+}\,\dif y,
\end{align*}
so \eqref{boundwithderivtop} follows from the resulting estimate
\[
|f(x)|\le\max\{|f(0)|,|f(1)|\}+\int_{0}^{1}(f'(y))^{+}\,\dif y\le\max\{|f(0)|,|f(1)|\}+\|(f'(y))^{+}\|_{\mathcal{C}([0,1])}.\qedhere
\]
\end{proof}
\begin{lem}
\label{lem:onedimlderiv}Suppose that $f:[-1,2]\to\mathbf{R}$ is
twice-differentiable. Then we have
\begin{equation}
\|f'\|_{\mathcal{C}([0,1])}\le\osc(f;[-1,2]\cap\mathbf{Z})+\tfrac{1}{2}\|(f'')^{+}\|_{\mathcal{C}([-1,2])}.\label{eq:onedimlosc-1}
\end{equation}
\end{lem}

\begin{proof}
Let $x_{*}\in[0,1]$ be such that $|f'(x_{*})|=\|f'\|_{\mathcal{C}([0,1])}$.
Assume without loss of generality that $f'(x_{*})\le0$; otherwise,
consider $\tilde{f}(x)=f(1/2-x)$. Then we have, for all $x\in[x_{*},2]$,
that
\[
f'(x)\le f'(x_{*})+\int_{x_{*}}^{x}(f'')^{+}(t)\,\dif t\le-\|f'\|_{\mathcal{C}([0,1])}+(x-x_{*})\|(f'')^{+}\|_{\mathcal{C}([-1,2])}.
\]
This implies that 
\[
\osc(f;[-1,2]\cap\mathbf{Z})\ge-(f(2)-f(1))=-\int_{1}^{2}f'(x)\,\dif x\ge\|f'\|_{\mathcal{C}([0,1])}-\tfrac{1}{2}\|(f'')^{+}\|_{\mathcal{C}([-1,2])},
\]
and \eqref{onedimlosc-1} follows.
\end{proof}
\begin{lem}
\label{lem:onedimlosc}Suppose that $f:[-1,2]\to\mathbf{R}$ is twice-differentiable.
Then we have
\begin{equation}
\osc(f;[0,1])\le\osc(f;[-1,2]\cap\mathbf{Z})+\tfrac{1}{2}\|(f'')^{+}\|_{\mathcal{C}([-1,2])}.\label{eq:onedimlosc}
\end{equation}
\end{lem}

\begin{proof}
Fix $0\le x_{1}<x_{2}\le1$ so that $|f(x_{2})-f(x_{1})|=\osc(f;[0,1])$.
The mean value theorem gives an $x_{*}\in(x_{1},x_{2})$ with
\[
|f'(x_{*})|=\frac{|f(x_{2})-f(x_{1})|}{x_{2}-x_{1}}\ge|f(x_{2})-f(x_{1})|=\osc(f;[0,1]),
\]
and then \lemref{onedimlderiv} yields \eqref{onedimlosc}.
\end{proof}

\subsection{Induction on the dimension\label{subsec:dimensioninduction}}

In order to carry out the inductive procedure, we first prove a lemma
controlling the oscillation of a function on the extrusion of a set
in a coordinate direction. Let $\mathbf{e}_{1},\ldots,\mathbf{e}_{d}$
be the standard basis of $\mathbf{R}^{d}$. In this section, we will
add sets $A,B\subset\mathbf{R}^{d}$ in the Minkowski sense, meaning
$A+B=\{a+b\mid a\in A,b\in B\}$, and also define the product of a
set $A\subset\mathbf{R}^{d}$ with a vector $\mathbf{v}$ as $A\mathbf{v}=\{a\mathbf{v}\mid a\in A\}$.
\begin{lem}
\label{lem:ddimlosc}Suppose that $f:\mathbf{R}^{d}\to\mathbf{R}$
is twice-differentiable, $i\in\{1,\ldots,d\}$, and $S\subset\mathbf{R}^{d}$.
Then, for any $\mathbf{x}\in S$,
\begin{equation}
\osc(f;S+[0,1]\mathbf{e}_{i})\le28\|(\partial_{ii}f)^{+}\|_{\mathcal{C}(S+[-1,2]\mathbf{e}_{i})}+3\max_{\ell=-1}^{2}[3|\partial_{i}f(\mathbf{x}+\ell\mathbf{e}_{i})|+2\osc(f;S+\ell\mathbf{e}_{i})].\label{eq:ddimlosc-1}
\end{equation}
\end{lem}

\begin{proof}
By \lemref{onedimlosc}, we have, for each $\mathbf{x}\in S$, that
\begin{equation}
\begin{aligned}\osc(f;\mathbf{x}+[0,1]\mathbf{e}_{i}) & \le\osc(f;\mathbf{x}+([-1,2]\cap\mathbf{Z})\mathbf{e}_{i})+\tfrac{1}{2}\|(\partial_{ii}f)^{+}\|_{\mathcal{C}(\mathbf{x}+[-1,2]\mathbf{e}_{i})}\\
 & \le\osc(f;S+([-1,2]\cap\mathbf{Z})\mathbf{e}_{i})+\tfrac{1}{2}\|(\partial_{ii}f)^{+}\|_{\mathcal{C}(S+[-1,2]\mathbf{e}_{i})}.
\end{aligned}
\label{eq:applyonedimlosc}
\end{equation}
By \lemref{combinemanyoscs}, we have
\begin{align}
\osc(f;S+[0,1]\mathbf{e}_{i}) & \le\osc(f;S)+2\sup_{\mathbf{x}\in S}\osc(f;\mathbf{x}+[0,1]\mathbf{e}_{i})\nonumber \\
 & \le\osc(f;S)+2\osc(f;S+([-1,2]\cap\mathbf{Z})\mathbf{e}_{i})+\|(\partial_{ii}f)^{+}\|_{\mathcal{C}(S+[-1,2]\mathbf{e}_{i})}\nonumber \\
 & \le3\osc(f;S+([-1,2]\cap\mathbf{Z})\mathbf{e}_{i})+\|(\partial_{ii}f)^{+}\|_{\mathcal{C}(S+[-1,2]\mathbf{e}_{i})},\label{eq:combinedifferentxs}
\end{align}
where in the second inequality we used \eqref{applyonedimlosc}. Also
by \lemref{combinemanyoscs}, we have, for any $\mathbf{x}\in S$,
that
\begin{equation}
\osc(f;S+([-1,2]\cap\mathbf{Z})\mathbf{e}_{i})\le\|\partial_{i}f\|_{L^{1}(\mathbf{x}+[-1,2]\mathbf{e}_{i};\mathcal{H}^{1})}+2\max_{\ell=-1}^{2}\osc(f;S+\ell\mathbf{e}_{i}),\label{eq:layercakebd}
\end{equation}
where $\mathcal{H}^{1}$ denotes the one-dimensional Hausdorff measure
on $\mathbf{x}+[-1,2]\mathbf{e}_{i}$. Using (a rescaled version of)
\lemref{boundwithderivtop}, we have
\begin{equation}
\|\partial_{i}f\|_{L^{1}(\mathbf{x}+[-1,2]\mathbf{e}_{i};\mathcal{H}^{1})}\le3\|\partial_{i}f\|_{\mathcal{C}(\mathbf{x}+[-1,2]\mathbf{e}_{i})}\le3(\max\{|\partial_{i}f(\mathbf{x}-\mathbf{e}_{i})|,|\partial_{i}f(\mathbf{x}+2\mathbf{e}_{i})|\}+3\|(\partial_{ii}f)^{+}\|_{\mathcal{C}(\mathbf{x}+[-1,2]\mathbf{e}_{i})}).\label{eq:partialL1bd}
\end{equation}
Plugging \eqref{partialL1bd} into \eqref{layercakebd}, and then
\eqref{layercakebd} into \eqref{combinedifferentxs}, we obtain \eqref{ddimlosc-1}
from the bound
\[
\osc(f;S+[0,1]\mathbf{e}_{i})\le9\max\{|\partial_{i}f(\mathbf{x}-\mathbf{e}_{i})|,|\partial_{i}f(\mathbf{x}+2\mathbf{e}_{i})|\}+28\|(\partial_{ii}f)^{+}\|_{\mathcal{C}(S+[-1,2]\mathbf{e}_{i})}+6\max_{\ell=-1}^{2}\osc(f;S+\ell\mathbf{e}_{i}).\qedhere
\]
\end{proof}
\lemref{ddimlosc} then forms the heart of the inductive step in the
following proposition.
\begin{prop}
\label{prop:oscbound-noninductive}For any twice-differentiable function
$f:\mathbf{R}^{d}\to\mathbf{R}$, we have
\[
\osc(f;[0,1]^{d})\le28d\left(\|\Delta^{+}f\|_{\mathcal{C}([-1,2]^{d})}+\|\nabla f\|_{\mathcal{C}(([-1,2]\cap\mathbf{Z})^{d})}\right).
\]
\end{prop}

\begin{proof}
We will prove by induction the stronger statement that for each $i\in\{0,\ldots,d\}$
and $\mathbf{y}\in([-1,2]\cap\mathbf{Z})^{d-i}$,
\begin{equation}
\max_{\mathbf{y}\in([-1,2]\cap\mathbf{Z})^{d-i}}\osc(f;[0,1]^{i}\times\{\mathbf{y}\})\le28i\left(\|\Delta^{+}f\|_{\mathcal{C}([-1,2]^{d})}+\|\nabla f\|_{\mathcal{C}(([-1,2]\cap\mathbf{Z})^{d})}\right).\label{eq:inductivething}
\end{equation}
The base case $i=0$ requires no explanation. For the inductive step,
we assume that
\begin{equation}
\max_{\mathbf{y}\in([-1,2]\cap\mathbf{Z})^{d-i+1}}\osc(f;[0,1]^{i-1}\times\{\mathbf{y}\})\le28(i-1)\left(\|\Delta^{+}f\|_{\mathcal{C}([-1,2]^{d})}+\max_{\mathbf{w}\in([-1,2]\cap\mathbf{Z})^{d}}|\nabla f(\mathbf{w})|]\right).\label{eq:inductivehypothesis}
\end{equation}
Now for all $\mathbf{z}\in([-1,2]\cap\mathbf{Z})^{d-i}$, we have
by \lemref{ddimlosc} that
\begin{align}
\osc(f;[0,1]^{i}\times\{\mathbf{z}\}) & \le28\|(\partial_{ii}f)^{+}\|_{\mathcal{C}([0,1]^{i-1}\times[-1,2]\times\{\mathbf{z}\})}+3\max_{\ell=-1}^{2}[3|\partial_{i}f(\mathbf{0}_{i-1},\ell,\mathbf{z})|+2\osc(f;[0,1]^{i-1}\times\{(\ell,\mathbf{z})\})].\label{eq:useddimlosc}
\end{align}
Here and henceforth, $\mathbf{0}_{i-1}$ denotes the vector of length
$i-1$ all of whose coordinates are $0$. We note that 
\begin{equation}
\max_{\ell=-1}^{2}|\partial_{i}f(\mathbf{0}_{i-1},\ell,\mathbf{z})|\le\|\nabla f\|_{\mathcal{C}(([-1,2]\cap\mathbf{Z})^{d})}\label{eq:boundthemaxthing}
\end{equation}
and, by the inductive hypothesis \eqref{inductivehypothesis}, we
have
\begin{equation}
\max_{\ell=-1}^{2}\osc(f;[0,1]^{i-1}\times\{(\ell,\mathbf{z})\})\le28(i-1)\left(\|\Delta^{+}f\|_{\mathcal{C}([-1,2]^{d})}+\|\nabla f\|_{\mathcal{C}(([-1,2]\cap\mathbf{Z})^{d})}\right).\label{eq:useinductivehypothesis}
\end{equation}
Plugging \eqref{boundthemaxthing} and \eqref{useinductivehypothesis}
into \eqref{useddimlosc}, we obtain
\[
\osc(f;[0,1]^{i}\times\{\mathbf{z}\})\le28i\|\Delta^{+}f\|_{\mathcal{C}([-1,2]^{d})}+(9+28(i-1))\|\nabla f\|_{\mathcal{C}(([-1,2]\cap\mathbf{Z})^{d})},
\]
which implies \eqref{inductivething}.
\end{proof}
It only remains to extract from our oscillation bound a bound on the
gradient of $f$.
\begin{prop}
If $f:\mathbf{R}^{d}\to\mathbf{R}$ is twice-differentiable, then
\begin{equation}
\|\nabla f\|_{\mathcal{C}([1/3,2/3]^{d})}\le84d(\|\Delta^{+}f\|_{\mathcal{C}([-1,2]^{d})}+\|\nabla f\|_{\mathcal{C}(([-1,2]\cap\mathbf{Z})^{d})}).\label{eq:derivbd}
\end{equation}
\end{prop}

\begin{proof}
Applying a rescaled version of \lemref{onedimlderiv} in each coordinate
direction, we obtain
\[
\tfrac{1}{3}\|\nabla f\|_{\mathcal{C}([1/3,2/3]^{d})}\le\osc(f;[0,1]^{d})+\tfrac{1}{18}\|\Delta^{+}f\|_{\mathcal{C}([0,1]^{d})}.
\]
Then we obtain \eqref{derivbd} by estimating, using \propref{oscbound-noninductive},
\begin{align*}
\|\nabla f\|_{\mathcal{C}([1/3,2/3]^{d})} & \le84d\left(\|\Delta^{+}f\|_{\mathcal{C}([-1,2]^{d})}+\|\nabla f\|_{\mathcal{C}(([-1,2]\cap\mathbf{Z})^{d})}\right)+\tfrac{1}{6}\|\Delta^{+}f\|_{\mathcal{C}([0,1]^{d})}.\qedhere
\end{align*}
\end{proof}
Now \propref{finallocalest} follows from translating and rescaling
\eqref{derivbd}.

\section{Proof of tightness\label{sec:tightness}}

\subsection{The periodic case}

To complete our tightness proof, we will use the maximum principle
described in \secref{maxprinc}, so it is simplest to first work in
the periodic setting and then let the periodicity go to infinity.
Throughout this subsection, we fix exponents
\begin{align}
m & \in(2d/(d+4),1), & \ell & \in(m/2,(1+2/d)m-1), & \eps & \in(0,\ell-m/2).\label{eq:defineexponents}
\end{align}
The requirement $d<4$ is necessary so that the interval $(2d/(d+4),1)$
is nonempty. We note that if $m$ is chosen arbitrarily in $(2d/(d+4),1)$,
then $\ell$ and $\eps$ can be chosen to satisfy \eqref{defineexponents}.
Let $\nu$ be the finite measure on $\mathbf{R}^{d}$ given by the
statement of \propref{deterministic-bound} with $\alpha=1/(1+\ell-m)$,
noting that $m>d/(2\alpha)$.
\begin{lem}
\label{lem:cutitininhalf}For every $M\in[1,\infty)$, there is a
$K_{0}=K_{0}(d,M)\in[1,\infty)$ so that if $K\ge K_{0}$, then there
is a $A=A(d,M,K)<\infty$ so that the following holds. Suppose that
$L\in[1,\infty)$ and $\psi\in\mathcal{C}([0,1];\mathcal{A}_{\xi}^{2})$
is $L\mathbf{Z}^{d}$-periodic. Let $v\in\mathcal{C}([0,1];\mathcal{A}_{\p_{m}})$
be an $L\mathbf{Z}^{d}$-periodic solution to \eqref{vPDE} and let
$u=v+\psi$. If
\begin{equation}
\int_{0}^{1}\|u(s,\cdot)\|_{L^{2}(\mathbf{R}^{d};\nu)}^{2}\,\dif s,\|\psi\|_{\mathcal{C}([0,1];\mathcal{C}_{\xi}^{2})},\|\psi\|_{\mathcal{C}([0,1];\mathcal{C}_{\p_{\eps}}^{2})}\le M\label{eq:ltM}
\end{equation}
then there exists a $t_{*}\in[0,\min\{K^{1-\ell}\|u(0,\cdot)\|_{\mathcal{C}_{\p_{m,K}}}^{-1},1\}]$
so that $\|u(t_{*},\cdot)\|_{\mathcal{C}_{\p_{m,K}}}\le\max\left\{ A,\tfrac{1}{2}\|u(0,\cdot)\|_{\mathcal{C}_{\p_{m,K}}}\right\} $.
\end{lem}

\begin{proof}
Let us first agree to take $A\ge K^{1-\ell}$, so we can assume that
\begin{equation}
\|u(0,\cdot)\|_{\mathcal{C}_{\p_{m,K}}}>K^{1-\ell},\label{eq:startingismax}
\end{equation}
as otherwise there is nothing to show. Define
\begin{align*}
Z(t) & =\max_{i=1}^{d}\|(\partial_{i}v_{i})^{+}\|_{\mathcal{C}_{\p_{\ell,K}}}; & z_{i}(t,x) & =\frac{\partial_{i}v_{i}(t,x)}{\p_{\ell,K}(x)}; & (i_{*}(t),x_{*}(t)) & \in\argmax_{(i,x)\in\{1,\ldots,d\}\times\mathbf{R}^{d}}(z_{i}(t,x),\partial_{t}z_{i}(t,x)),
\end{align*}
where in the $\argmax$ we use the lexicographic ordering on $\mathbf{R}\times\mathbf{R}$.
In other words, $(i_{*}(t),x_{*}(t))$ is chosen first to maximize
$z_{i}(t,x)$, and then, among the maximizers of $z_{i}(t,x)$, to
maximize $\partial_{t}z_{i}(t,x)$. The $\argmax$ is guaranteed to
exist since $v$ is assumed to be $L\mathbf{Z}^{d}$-periodic. Note
that $Z(t)=z_{i_{*}}(t,x_{*}(t))$. 

By \propref{dtzbound2}, \eqref{ltM} and \cite[Lemma 3.5]{Ham86},
for almost every $t$ such that $Z(t)\ge4(d+3)+M$, we have
\begin{equation}
Z'(t)\le-\tfrac{1}{8}\p_{\ell,K}(x_{*}(t))Z(t)^{2}+|u(t,x_{*}(t))|\left(\frac{Z(t)}{\p_{1,K}(x_{*}(t))}+\frac{M}{\p_{\ell-\eps,K}(x_{*}(t))}\right).\label{eq:Zprimeready}
\end{equation}
Now \propref{deterministic-bound} and \eqref{ltM} imply that there
is a constant $C$, depending only on the dimension $d$, so that
\begin{equation}
\p_{m,K}(x_{*}(t))^{-1}|u(t,x_{*}(t))|\le C(Z(t)+\|u(t,\cdot)\|_{L^{2}(\mathbf{R}^{d};\nu)}+M).\label{eq:boundatthistime}
\end{equation}
Also, \propref{ugrowth} and \eqref{startingismax} imply that there
is a $K_{1}=K_{1}(M)$ so that if $K\ge K_{0}$ and $t<K^{1-\ell}\|u(0,\cdot)\|_{\mathcal{C}_{\p_{m,K}}}^{-1}$
then
\[
|u(t,x_{*}(t))|\le\left([\|u(0,\cdot)\|_{\mathcal{C}_{\p_{m,K}}}^{-1}-K^{-(1-\ell)}t]^{-1}+M\right)\p_{m+\eps t,K}(x_{*}(t)).
\]
In particular, if $t\le t_{0}\coloneqq\tfrac{1}{2}K^{1-\ell}\|u(0,\cdot)\|_{\mathcal{C}_{\p_{m,K}}}^{-1}$,
then we have
\begin{equation}
|u(t,x_{*}(t))|\le\left(2\|u(0,\cdot)\|_{\mathcal{C}_{\p_{m,K}}}+M\right)\p_{m+\eps t,K}(x_{*}(t)).\label{eq:Debound-1}
\end{equation}

Using \eqref{boundatthistime} and \eqref{Debound-1} in \eqref{Zprimeready},
if $K\ge K_{1}$ we obtain, for almost every $t\in[0,t_{0}]$ such
that $Z(t)\ge4(d+3)+M$, that
\begin{equation}
\begin{aligned}Z'(t) & \le-\tfrac{1}{8}\p_{\ell,K}(x_{*}(t))Z(t)^{2}+C[Z(t)+\|u(t,\cdot)\|_{L^{2}(\mathbf{R}^{d};\nu)}+M]Z(t)\p_{m-1,K}(x_{*}(t))\\
 & \qquad\qquad+[2\|u(0,\cdot)\|_{\mathcal{C}_{\p_{m,K}}}+M]M\p_{m+\eps t-\ell+\eps,K}(x_{*}(t)).
\end{aligned}
\label{eq:Zmess}
\end{equation}
Since $m<1$ and $m+\eps t-\ell+\eps\le m-\ell+2\eps<\ell$ by \eqref{defineexponents}
(recalling that $t\le t_{0}\le1$), there is a $K_{2}=K_{2}(d,M)<\infty$
so that if $K\ge K_{2}$ then we have, for all $x\in\mathbf{R}^{d}$,
\begin{align*}
C(1+M)\p_{m-1,K}(x) & ,M(2+M)\p_{m+\eps t-\ell+\eps,K}(x)\le\tfrac{1}{32}\p_{\ell,K}(x)\le1.
\end{align*}
Using these estimates in \eqref{Zmess} yields, as long as $K\ge\max\{K_{1},K_{2}\}$
,
\begin{align*}
Z'(t) & \le-\tfrac{1}{16}\p_{\ell,K}(x_{*}(t))[Z(t)^{2}-\tfrac{1}{2}\|u(0,\cdot)\|_{\mathcal{C}_{\p_{m,K}}}]+\|u(t,\cdot)\|_{L^{2}(\mathbf{R}^{d};\nu)}Z(t)\le-\tfrac{1}{32}K^{\ell}Z(t)^{2}+\|u(t,\cdot)\|_{L^{2}(\mathbf{R}^{d};\nu)}Z(t)
\end{align*}
for almost every $t$ such that $Z(t)\ge\max\{4(d+3)+M,\|u(0,\cdot)\|_{\mathcal{C}_{\p_{m,K}}}^{1/2}\}$.
So by \lemref{riccati-with-linear} below we have, for all $t\le t_{0}$,
\begin{align*}
Z(t) & \le\exp\left\{ \int_{0}^{t}\|u(t,\cdot)\|_{L^{2}(\mathbf{R}^{d};\nu)}\,\dif t\right\} \max\left\{ 8(d+3)+2M,2\|u(0,\cdot)\|_{\mathcal{C}_{\p_{m,K}}}^{1/2},\left[Z(0)^{-1}+\tfrac{1}{64}K^{\ell}t\right]^{-1}\right\} \\
 & \le\e^{M}\max\left\{ 8(d+3)+2M,2\|u(0,\cdot)\|_{\mathcal{C}_{\p_{m,K}}}^{1/2},64(K^{\ell}t)^{-1}\right\} ,
\end{align*}
Now if $t\in[t_{0}/2,t_{0}]$, then by the definition of $t_{0}$,
we have $t\ge\frac{1}{4}K^{1-\ell}\|u(0,\cdot)\|_{\mathcal{C}_{\p_{m,K}}}^{-1}$,
so
\[
Z(t)\le\e^{M}\max\left\{ 8(d+3)+2M,2\|u(0,\cdot)\|_{\mathcal{C}_{\p_{m,K}}}^{1/2},256K^{-1}\|u(0,\cdot)\|_{\mathcal{C}_{\p_{m,K}}}\right\} .
\]
Next we observe, using \eqref{ltM} and the definition of $t_{0}$,
that
\begin{align*}
\frac{1}{t_{0}/2}\int_{t_{0}/2}^{t_{0}}\|u(t,\cdot)\|_{L^{2}(\mathbf{R}^{d};\nu)}\,\dif t\le\left(\frac{2}{t_{0}}\int_{t_{0}/2}^{t_{0}}\|u(t,\cdot)\|_{L^{2}(\mathbf{R}^{d};\nu)}^{2}\,\dif t\right)^{1/2} & \le\sqrt{2M/t_{0}}\le2\sqrt{M}K^{-\frac{1-\ell}{2}}\|u(0,\cdot)\|_{\mathcal{C}_{\p_{m,K}}}^{1/2}.
\end{align*}
Thus we have a $t_{*}\in[t_{0}/2,t_{0}]$ so that
\[
\|u(t_{*},\cdot)\|_{L^{2}(\mathbf{R}^{d};\nu)}\le2\sqrt{M}K^{-\frac{1-\ell}{2}}\|u(0,\cdot)\|_{\mathcal{C}_{\p_{m,K}}}^{1/2}\le2\sqrt{M}\|u(0,\cdot)\|_{\mathcal{C}_{\p_{m,K}}}^{1/2}
\]
Then we have, again using \propref{deterministic-bound},
\begin{multline}
\|u(t_{*},\cdot)\|_{\mathcal{C}_{\p_{m,K}}}\le C(\|u(t_{*},\cdot)\|_{L^{2}(\mathbf{R}^{d};\nu)}+Z(t_{*})+M)\\
\le C\left(2\sqrt{M}\|u(0,\cdot)\|_{\mathcal{C}_{\p_{m,K}}}^{1/2}+\e^{M}\max\left\{ 8(d+3)+2M,2\|u(0,\cdot)\|_{\mathcal{C}_{\p_{m,K}}}^{1/2},256K^{-1}\|u(0,\cdot)\|_{\mathcal{C}_{\p_{m,K}}}\right\} +M\right).\label{eq:penultimateutstarCpmkbound}
\end{multline}
Let $K\ge2048\e^{M}$, so we have an $A_{1}=A_{1}(d,K,M)$ so that
if $\|u(0,\cdot)\|_{\mathcal{C}_{\p_{m,K}}}\ge A_{1}$, then \eqref{penultimateutstarCpmkbound}
implies $\|u(t_{*},\cdot)\|_{\mathcal{C}_{\p_{m,K}}}\le\tfrac{1}{2}\|u(0,\cdot)\|_{\mathcal{C}_{\p_{m,K}}}$.
This completes the proof, since we can take $K_{0}=\max\{K_{1},K_{2},2048\e^{M}\}$
and $A=\max\{K^{1-\ell},A_{1}\}$ (recalling the original agreement
leading to \eqref{startingismax}).
\end{proof}
It remains to prove the Riccati equation estimate we used in the proof
of \lemref{cutitininhalf}.
\begin{lem}
\label{lem:riccati-with-linear}Suppose that $T>0$, $h:[0,T]\to\mathbf{R}$
is Lipschitz, $a>0$, $f\in L^{1}([0,T])$, $b>0$, and for almost
every $t\in[0,T]$ such that $h(t)\ge b$, we have $h'(t)\le-ah(t)^{2}+f(t)h(t)$.
Then we have
\begin{equation}
h(t)\le\exp\left\{ \|f\|_{L^{1}([0,t])}\right\} \max\left\{ 2b,[h(0)^{-1}+at/2]^{-1}\right\} .\label{eq:riccati-with-linear-concl}
\end{equation}
\end{lem}

\begin{proof}
Assume without loss of generality that $f(t)\ge0$ for all $t$. Let
$F(t)=\exp\left\{ -\int_{0}^{t}f(s)\,\dif s\right\} $ and let $H(t)=F(t)h(t)$.
Note in particular that $F(t)\le1$ for all $t$. Then, for almost
every $t\in[0,T]$ such that $H(t)\ge b$, we have $h(t)\ge b$ and
thus %
\[
H'(t)\le-\omega(t)F(t)h(t)^{2}\le-aH(t)^{2}.
\]
Then \eqref{riccati-with-linear-concl} follows from standard bounds
on solutions to the Riccati equation:
\[
\exp\left\{ -\|f\|_{L^{1}([0,t])}\right\} h(t)=H(T)\le\max\left\{ 2b,[H(0)^{-1}+aT/2]^{-1}\right\} =\max\left\{ 2b,[h(0)^{-1}+aT/2]^{-1}\right\} .\qedhere
\]
\end{proof}
We now iterate \lemref{cutitininhalf} to show that our dynamics brings,
in time $1$, any initial condition down to a size depending only
on the forcing on $[0,2]$ and the integrated $L^{2}$ norm $\int_{0}^{2}\|u(s,\cdot)\|_{L^{2}(\mathbf{R}^{d};\nu)}^{2}\,\dif s$.
The last quantity, when $u$ is taken to solve \eqref{uLPDE} with
initial condition $0$, will be \emph{a priori} bounded using \propref{L2bound}.
\begin{prop}
\label{prop:bringitdown}For every $M\in[1,\infty)$, there is a $B=B(d,M)<\infty$
so that the following holds. Suppose that $L\in[1,\infty)$ and $\psi\in\mathcal{C}([0,2];\mathcal{A}_{\xi}^{2})$
is $L$-periodic. Let $v$ be an $L$-periodic solution to \eqref{vPDE}
and let $u=v+\psi$. If
\begin{equation}
\int_{0}^{2}\|u(s,\cdot)\|_{L^{2}(\mathbf{R}^{d};\nu)}^{2}\,\dif s,\|\psi\|_{\mathcal{C}([0,2];\mathcal{C}_{\xi}^{2})},\|\psi\|_{\mathcal{C}([0,2];\mathcal{C}_{\p_{\eps}}^{2})}\le M\label{eq:ltM-1}
\end{equation}
then there exists a $t\in[0,2]$ so that $\|u(t,\cdot)\|_{\mathcal{C}_{\p_{m}}}\le B$.
\end{prop}

\begin{proof}
Choose $K_{0}$ depending on $M$ as in \lemref{cutitininhalf} and
fix $K\ge K_{0}$. (The choice $K=K_{0}$ is fine.) We inductively
define a finite sequence of times $t_{0},t_{1},\ldots,t_{N}$ as follows.
Let $t_{0}=0$. If $t_{k}$ has been defined and either $t_{k}\ge1$
or $\|u(t_{N},\cdot)\|_{\mathcal{C}_{\p_{m,K}}}\le A$, then let $N=k$
and stop. Otherwise, by \lemref{cutitininhalf}, since \eqref{ltM-1}
implies \eqref{ltM} for the equation considered as starting at time
$t_{k}<1$, there is an $A=A(d,K,M)<\infty$ so that we can find a
\begin{equation}
t_{k+1}\in[t_{k},t_{k}+K^{1-\ell}\|u(t_{k},\cdot)\|_{\mathcal{C}_{\p_{m,K}}}^{-1}]\label{eq:tknotsobig}
\end{equation}
such that $\|u(t_{k+1},\cdot)\|_{\mathcal{C}_{\p_{m,K}}}\le\max\left\{ A,\tfrac{1}{2}\|u(t_{k},\cdot)\|_{\mathcal{C}_{\p_{m,K}}}\right\} $.%
{} Thus, for each $0\le k\le N-1$, we have 
\[
\|u(t_{k},\cdot)\|_{\mathcal{C}_{\p_{m,K}}}\ge2^{N-k}\|u(t_{N},\cdot)\|_{\mathcal{C}_{\p_{m,K}}}.
\]
In the case when $\|u(t_{N},\cdot)\|_{\mathcal{C}_{\p_{m,K}}}>A$
and thus $t_{N}\ge1$, this means that by \eqref{tknotsobig} we get
\[
1\le t_{N}=\sum_{k=1}^{N}(t_{k}-t_{k-1})\le K^{1-\ell}\sum_{k=0}^{N-1}\|u(t_{k},\cdot)\|_{\mathcal{C}_{\p_{m,K}}}^{-1}\le K^{1-\ell}\|u(t_{N},\cdot)\|_{\mathcal{C}_{\p_{m,K}}}^{-1}\sum_{k=0}^{N-1}2^{-(N-k)}\le K^{1-\ell}\|u(t_{N},\cdot)\|_{\mathcal{C}_{\p_{m,K}}}^{-1},
\]
so $\|u(t_{N},\cdot)\|_{\mathcal{C}_{\p_{m,K}}}\le K^{1-\ell}$. Since
the alternative is that $\|u(t_{N},\cdot)\|_{\mathcal{C}_{\p_{m,K}}}\le A$,
in either case we have
\[
\|u(t_{N},\cdot)\|_{\mathcal{C}_{\p_{m}}}\le K^{m}\|u(t_{N},\cdot)\|_{\mathcal{C}_{\p_{m,K}}}\le(1+K)^{m}\max\{A,K^{1-\ell}\}.\qedhere
\]
\end{proof}
Up until now all of our arguments have been agnostic to the initial
conditions. The value of this, of course, is that they can be applied
to the equation started at any time. The next proposition shows how
we carry this out. %

\begin{prop}
\label{prop:bounded-periodic}Let $d<4$. For each $\delta>0$, we
have a constant $Q(\delta)<\infty$ so that the following holds. Let
$L\in[1,\infty)$ and let $u$ be the solution to \eqref{uLPDE} with
$u(0,\cdot)\equiv0$. Also, let $T\ge0$ and let $S_{T}\sim\Uniform([0,T])$
be independent of everything else. Then $\mathbf{P}(\|u(3+S_{T},\cdot)\|_{\mathcal{C}_{\p_{m}}}>Q(\delta))<\delta.$
\end{prop}

\begin{proof}
Let $\psi$ solve \eqref{ASHEL} with $\psi(0,\cdot)\equiv0$. Standard
Gaussian process estimates, using the smoothness of the mollifier
$\rho$, show that there is a constant $C<\infty$, independent of
$L$, so that $\mathbf{E}\|\psi\|_{\mathcal{C}([t,t+3];\mathcal{C}_{\xi}^{2})},\mathbf{E}\|\psi\|_{\mathcal{C}([t,t+3];\mathcal{C}_{\p_{\eps}}^{2})}\le C$
for each $t\ge0$. Also, by \propref{L2bound} and Fubini's theorem,
using the fact that $\nu$ is a finite measure, we have a (different)
constant $C<\infty$, also independent of $L$, so that $\frac{1}{t}\int_{0}^{t}\mathbf{E}\|u(s,\cdot)\|_{L^{2}(\mathbf{R}^{d};\nu)}^{2}\,\dif s\le C$
for each $t\ge0$. Markov's inequality thus implies that there is
a $Q_{0}(\delta)<\infty$, independent of $L$, so that
\[
\mathbf{P}\left(\max\left\{ \|\psi\|_{\mathcal{C}([\lfloor S_{T}\rfloor,\lfloor S_{T}\rfloor+3];\mathcal{C}_{\xi}^{2})},\|\psi\|_{\mathcal{C}([\lfloor S_{T}\rfloor,\lfloor S_{T}\rfloor+3];\mathcal{C}_{\p_{\eps}}^{2})},\frac{1}{2}\int_{\lfloor S_{T}\rfloor}^{\lfloor S_{T}\rfloor+2}\|u(s,\cdot)\|_{L^{2}(\mathbf{R}^{d};\nu)}^{2}\,\dif s\right\} >Q_{0}(\delta)\right)<\delta.
\]
By \propref{bringitdown}, since $u-\psi$ is $L\mathbf{Z}^{d}$-periodic
and solves \eqref{vPDE}, this means that there is a $Q_{1}(\delta)<\infty$,
still independent of $L$, so that
\[
\mathbf{P}\left(\text{there is a \ensuremath{t\in[\lfloor S_{T}\rfloor,\lfloor S_{T}\rfloor+2]} so that \ensuremath{\|u(t,\cdot)\|_{\mathcal{C}_{\p_{m}}}\le Q_{1}(\delta)}}\right)\ge1-\delta.
\]
Thus by \thmref{existenceuniqueness}, there is a $Q_{2}(\delta)<\infty$,
not depending on $L$, so that $\mathbf{P}\left(\text{\ensuremath{\|u(S_{T}+3,\cdot)\|_{\mathcal{C}_{\p_{m}}}\le Q_{2}(\delta)}}\right)\ge1-\delta$.
\end{proof}

\subsection{The non-periodic case}

All of the estimates above have been uniform in the periodicity length
$L$, so it should not be surprising that they hold for $L=\infty$
as well. For each $L\in[1,\infty]$ let $u^{[L]}$ be a solution to
\eqref{uLPDE} with initial condition $u^{[L]}(0,\cdot)\equiv0$.
Our first lemma passes to the limit $L\to\infty$ in \propref{bounded-periodic},
showing that the family $(u^{[\infty]}(S_{T},\cdot))_{T\ge0}$ is
bounded in probability with respect to the $\mathcal{C}_{\p_{\omega}}$
norm whenever $\omega\in(2d/(d+4),1)$. Since the equation \eqref{Burgers}
is well-posed in this space, we can then upgrade this boundedness
in probability to tightness by using the parabolic regularity statement
in \thmref{existenceuniqueness}, which we do in \propref{tightness}
below.
\begin{lem}
\label{lem:avgbddCpmprime}If $d<4$, then for any $\omega\in(2d/(d+4),1)$
and any $\delta>0$, we have a constant $Q(\delta)<\infty$ so that,
for all $T>0$, if $S_{T}\sim\Uniform([0,T])$ is independent of everything
else, we have
\begin{equation}
\mathbf{P}\left(\|u^{[\infty]}(S_{T}+3,\cdot)\|_{\mathcal{C}_{\p_{\omega}}}\ge Q(\delta)\right)\le\delta.\label{eq:bddinprob}
\end{equation}
\end{lem}

\begin{proof}
By \eqref{continuity} of \thmref{existenceuniqueness} and \eqref{VLconvtoV}
of \lemref{VLdef}, for fixed $t>0$ and $M>0$, $\|u^{[L]}-u^{[\infty]}\|_{\mathcal{C}([0,T];\mathcal{C}([-M,M]^{d}))}$
approaches $0$ as $L$ goes to infinity, almost surely. %
This means that for fixed $T>0$ and $M>0$, we almost surely have
\begin{equation}
\lim_{L\to\infty}\|u^{[L]}(S_{T}+3,\cdot)-u^{[\infty]}(S_{T}+3,\cdot)\|_{\mathcal{C}([-M,M]^{d})}=0.\label{eq:convalmostsure}
\end{equation}

Choose $m\in(2d/(d+4),\omega)$ and choose $\ell,\eps$ to so that
$m,\ell,\eps$ satisfy \eqref{defineexponents}. \propref{bounded-periodic}
implies that for each $\delta>0$ and $n\in\mathbf{N}$, there is
a $Q_{n}(\delta)<\infty$ so that
\begin{equation}
\sup_{L\in[1,\infty)}\mathbf{P}\left(\|u^{[L]}(S_{T}+3,\cdot)\|_{\mathcal{C}_{\p_{m}}}>Q_{n}(\delta)\right)=\delta/2^{n+1}.\label{eq:specializeMarkov}
\end{equation}
By \eqref{convalmostsure}, for fixed choice of $M_{n}$, we can choose
$L_{n}$ so large that
\begin{equation}
\mathbf{P}\left(\|u^{[L_{n}]}(S_{T}+3,\cdot)-u(S_{T}+3,\cdot)\|_{\mathcal{C}([-M_{n},M_{n}]^{d})}\ge1\right)\le\delta/2^{n+1}.\label{eq:uselimit}
\end{equation}
Combining \eqref{specializeMarkov} and \eqref{uselimit} with a union
bound, and then union bounding over all $n$, we have%
\begin{equation}
\mathbf{P}\left(\sup_{n\in\mathbf{N}}\frac{\|u^{[\infty]}(S_{T}+3,\cdot)\mathbf{1}_{[-M_{n},M_{n}]^{d}}\|_{\mathcal{C}_{\p_{m}}}}{1+Q_{n}(\delta)}\ge1\right)\le\delta\label{eq:unionovern}
\end{equation}
for any sequence $\{M_{n}\}_{n}$. Now define $M_{0}(\delta)=0$ and
choose $(M_{n}(\delta))_{n\in\mathbf{N}}$ so that $M_{n}(\delta)$
is nondecreasing in $n$ and there is a constant $\tilde{Q}(\delta)$
so that $1+Q_{n}(\delta)\le\tilde{Q}(\delta)\frac{\p_{\omega}(M_{n-1}(\delta))}{\p_{m}(M_{n-1}(\delta))}$
for each $n$. If we define $A_{n}(\delta)$ to be the square annulus
$[-M_{n}(\delta),M_{n}(\delta)]^{d}\setminus[-M_{n-1}(\delta),M_{n-1}(\delta)]^{d}$,
then we have
\begin{align*}
\|u^{[\infty]}(S_{T}+3,\cdot)\|_{\mathcal{C}_{\p_{\omega}}} & =\adjustlimits\sup_{n\in\mathbf{N}}\sup_{x\in A_{n}(\delta)}\frac{|u^{[\infty]}(S_{T}+3,x)|}{\p_{\omega}(x)}\le\adjustlimits\sup_{n\in\mathbf{N}}\sup_{x\in A_{n}(\delta)}\frac{|u^{[\infty]}(S_{T}+3,x)|}{\p_{m}(x)}\frac{\p_{m}(M_{n-1}(\delta))}{\p_{\omega}(M_{n-1}(\delta))}\\
 & \le\tilde{Q}(\delta)\adjustlimits\sup_{n\in\mathbf{N}}\sup_{x\in[-M_{n}(\delta),M_{n}(\delta)]^{d}}\frac{|u^{[\infty]}(S_{T}+3,x)|}{(1+Q_{n}(\delta))\p_{m}(x)}\le\tilde{Q}(\delta)\sup_{n\in\mathbf{N}}\frac{\|u^{[\infty]}(S_{T}+3,\cdot)\mathbf{1}_{[-M_{n}(\delta),M_{n}(\delta)]^{d}}\|_{\mathcal{C}_{\p_{m}}}}{1+Q_{n}(\delta)}.
\end{align*}
Combining this with \eqref{unionovern}, we obtain \eqref{bddinprob}.
\end{proof}
\begin{prop}
\label{prop:tightness}For each $T\ge0$, let $S_{T}\sim\Uniform([0,T])$
be independent of everything else. If $d<4$, for any $\omega\in(2d/(d+4),1)$,
the family $(u^{[\infty]}(S_{T}+4,\cdot))_{T\ge0}$ is tight in the
topology of $\tilde{\mathcal{A}}_{\p_{\omega}}$ (defined in \eqref{tildeAw}).
\end{prop}

\begin{proof}
Choose $\omega',\omega''$ so that $2d/(d+4)<\omega'<\omega''<\omega$.
By \lemref{avgbddCpmprime}, for each $\delta>0$ we have a constant
$Q(\delta)<\infty$ so that
\[
\mathbf{P}\left(\|u^{[\infty]}(S_{T}+3,\cdot)\|_{\mathcal{C}_{\p_{\omega''}}}\ge Q(\delta)\right)\le\delta
\]
for all $T\ge0$. In light of \thmref{existenceuniqueness}, this
means that there is an $\alpha>0$ (independent of $\delta$) and
a $Q'=Q'(Q(\delta))<\infty$ so that
\[
\mathbf{P}\left(\|u^{[\infty]}(S_{T}+4,\cdot)\|_{\mathcal{C}_{\p_{\omega'}}}+\|u^{[\infty]}(S_{T}+4,\cdot)\|_{\mathcal{C}_{\p_{2\omega'}}^{\alpha}}\ge Q'\right)\le\delta.
\]
This means that the random variable $u^{[\infty]}(S_{T}+4,\cdot)$
is bounded in probability in the topology of $\mathcal{A}_{\p_{\omega'}}\cap\mathcal{A}_{\p_{2\omega'}}^{\alpha}$.
But this space embeds compactly into $\mathcal{A}_{\p_{\omega}}$
by \lemref{weightedcompactness}, and indeed into $\tilde{\mathcal{A}}_{\p_{\omega}}$,
so the proposition is proved.
\end{proof}
As in \cite{DGR19}, we can now show that stationary solutions exist
for \eqref{Burgers}, using a Krylov--Bogoliubov argument (see e.g.
\cite{DPZ96}).
\begin{proof}[Proof of \thmref{maintheorem}.]
Fix $\omega\in(2d/(d+4),1)$. Let $u$ be a solution to \eqref{Burgers}
with initial condition $u(0,\cdot)\equiv0$. For each $T\ge0$, let
$S_{T}\sim\Uniform([0,T])$ be independent of everything else. By
\propref{tightness} and Prokhorov's theorem, for any $\omega\in(2d/(d+4),1)$
we have an increasing sequence $T_{k}\uparrow\infty$ and a random
$\underline{u}^{*}\in\tilde{\mathcal{A}}_{\p_{\omega}}\subset\mathcal{A}_{\p_{\omega}}$
so that
\begin{equation}
u(S_{T_{k}}+4,\cdot)\xrightarrow[k\to\infty]{\mathrm{law}}\underline{u}^{*}\label{eq:convinlaw}
\end{equation}
with respect to the topology of $\mathcal{C}_{\p_{\omega}}$. It is
not difficult to see that $\underline{u}^{*}$ is space-stationary
in law. Let $u^{*}$ be a solution to \eqref{Burgers} with initial
condition $u^{*}(0,\cdot)=\underline{u}^{*}$. Fix $M>0$ and let
$f\in\mathcal{C}(\mathcal{C}([-M,M]^{d};\mathbf{R}^{d}))$, the space
of bounded continuous functions on $\mathcal{C}([-M,M]^{d};\mathbf{R}^{d})$.
By the Skorokhod representation theorem, we can find a family $(\underline{u}^{[k]})_{k\in\mathbf{N}}$
of random variables on the same probability space as $\underline{u}^{*}$,
taking values in $\tilde{\mathcal{A}}_{\p_{\omega}}$, so that $\Law(\underline{u}^{[k]})=\Law(u(S_{T_{k}}+4,\cdot))$
and $\lim\limits _{k\to\infty}\|\underline{u}^{[k]}-\underline{u}^{*}\|_{\mathcal{C}_{\p_{\omega}}}=0$
almost surely. By the bounded convergence theorem and the fact that
$f\in\mathcal{C}(\mathcal{C}([-M,M]^{d};\mathbf{R}^{d}))\subset\mathcal{C}(\mathcal{C}_{\p_{\omega}}(\mathbf{R}^{d};\mathbf{R}^{d}))$,
this means that
\begin{equation}
\lim_{k\to\infty}\mathbf{E}|f(\underline{u}^{[k]})-f(\underline{u}^{*})|=0.\label{eq:fukfustar}
\end{equation}
Now fix $t>0$ and let $u^{[k]}$ be a solution to \eqref{Burgers}
with initial condition $u^{[k]}(0,\cdot)=\underline{u}^{[k]}$. By
\eqref{fukfustar}, \thmref{existenceuniqueness}, and the bounded
convergence theorem, we have
\begin{equation}
\lim_{k\to\infty}\mathbf{E}|f(u^{[k]}(t,\cdot))-f(u^{*}(t,\cdot))|=0.\label{eq:applycontinuity}
\end{equation}
On the other hand, by construction we have $\Law(u^{[k]}(t,\cdot))=\Law(u(S_{T_{k}}+4+t,\cdot))$.
Thus, for fixed $t$, we have
\begin{align*}
|\mathbf{E}f(u^{*}(0,\cdot))-\mathbf{E}f(u^{*}(t,\cdot))| & \le|\mathbf{E}f(u^{*}(0,\cdot))-\mathbf{E}f(u(S_{T_{k}}+4,\cdot))|+|\mathbf{E}f(u(S_{T_{k}}+4,\cdot))-\mathbf{E}f(u(S_{T_{k}}+4+t,\cdot))|\\
 & \qquad+|\mathbf{E}f(u^{[k]}(t,\cdot))-\mathbf{E}f(u^{*}(t,\cdot))|.
\end{align*}
The first term on the right-hand side goes to $0$ as $k\to\infty$
by \eqref{convinlaw}, the second by a simple coupling argument and
the bounded convergence theorem, and the third by \eqref{applycontinuity}.
This means that $\mathbf{E}f(u^{*}(0,\cdot))=\mathbf{E}f(u^{*}(t,\cdot))$
for all $f\in\mathcal{C}(\mathcal{C}([-M,M]^{d};\mathbf{R}^{d}))$
for any $M>0$, hence for all $f\in\mathcal{C}(\mathcal{C}_{\p_{\ell}}(\mathbf{R}^{d};\mathbf{R}^{d}))$
by approximation. Thus, $u^{*}$ is statistically stationary in time.
\end{proof}
\printbibliography

\end{document}